\documentclass[12pt]{amsart}
\usepackage{amsfonts,amssymb,amscd,xy}
\usepackage[leqno]{amsmath}
\usepackage{mathrsfs}
\usepackage{amssymb}
\usepackage{amsmath}
\usepackage{amsxtra}
\usepackage{amsthm}
\usepackage{hyperref}
\usepackage{rotating}
\usepackage{anyfontsize}
\fontsize{1}{2}

\newcommand{\Hom}{{\mathrm{Hom}}}

\def\s{\mathscr }

\DeclareMathAlphabet{\mathbbmsl}{U}{bbm}{m}{sl}

\newcommand{\res}{{\rm res}}
\newcommand{\tr}{{\rm tr}}
\newcommand{\nat}{\le\natural}

\def\uxs{U_{\be X\be/\lbe S}\le}
\def\uys{U_{\be Y\be/\lbe S}\le}
\def\uxys{U_{\lbe X\be\times_{\be S}\lbe Y\be/\lbe S}\le}
\def\uxss{U_{\be S}\lbe(\be X\le)}

\def\nh3{{\rm NH}^{3}\lbe}

\def\ruxs{{\rm RU}_{\be X\be/S}}

\def\rdiv{{\rm{RDiv}}_{\be X\be/\be S}}

\def\pic{{\rm{Pic}}\,}

\def\upicxs{\be{\rm{UPic}}_{\lbe X\be/\lbe S}}
\def\upicys{{\rm{UPic}}_{\e Y\be/\le S}}

\def\upicxys{{\rm{UPic}}_{\lbe X\lbe\times_{\be S}\lbe Y\be/\lbe S}}

\def\bg{{\mathbb G}}
\def\g{\varGamma}

\def\s{\mathscr }
\def\tors{\lle\rm{tors}}

\def\dbs{D^{\e b}\lbe(S_{\et})}

\def\tors{\lbe{\rm tors}}

\newcommand{\isoto}{\overset{\!\sim}{\to}}

\def\g{\varGamma}

\def\br{{\rm{Br}}\e}
\def\brp{{\rm{Br}}^{\e\prime}\lbe\lbe}
\def\bra{{\rm{Br}}_{\lbe\rm{a}}^{\le\prime}\lbe(\lbe X\!/\be S\le)}
\def\bray{{\rm{Br}}_{\lbe\rm{a}}^{\le\prime}\lbe(\le Y\!/\be S\e)}
\def\braxy{{\rm{Br}}_{\lbe\rm{a}}^{\le\prime}\lbe(X\!\times_{\be S}\! Y\!/\lbe S\e)}
\def\braxk{{\rm{Br}}_{\lbe\rm{a}}^{\le\prime}(X/k\e)}

\def\bro{{\rm{Br}}_{\lbe 1}^{\le\prime}\lbe\lle(\lbe X\!/\be S\e)}
\def\broys{{\rm{Br}}_{\lbe 1}^{\le\prime}(\le Y\!/\be S\e)}
\def\broxys{{\rm{Br}}_{\lbe 1}^{\le\prime}\lbe(\lbe X\!\times_{\be S}\!Y\!/\be S\le)}

\def\brt{{\rm{Br}}_{\lbe 2}^{\le\prime}\lbe(\lbe X\!/\be S\e)}
\def\brty{{\rm{Br}}_{2}^{\le\prime}\lbe(Y\be/\be S\e)}
\def\brtxy{{\rm{Br}}_{2}^{\le\prime}\lbe(\lbe X\!\be\times_{\be S}\!\lbe Y\!/\be S\e)}

\def\uys{U_{\le Y/S}\le}

\def\picxs{{\rm Pic}_{\lbe X\!/\lbe S}}
\def\picys{{\rm Pic}_{\e Y\!\lbe/\lbe S}}
\def\picxys{{\rm Pic}_{X\lbe\times_{\be S}\lbe Y\lbe/\lbe S}}
\def\npicxs{{\rm N}\lle\pic\be(\lbe X\be/\be S)}

\def\npic{{\rm N}\lle\pic\lle}
\def\nbr{{\rm N}\lle\brp}

\def\brxs{{\rm{Br}}^{\le\prime}_{\be X\be/\lbe S}}
\def\brys{{\rm Br}^{\le\prime}_{Y\!/\lbe S}}
\def\brxys{{\rm Br}^{\le\prime}_{\be X\be\times_{\be S}\lbe Y\!/\lbe S\le}}

\def\nbrxs{{\nbr}(\be X\!/\be S)}
\def\nbrys{{\nbr}(\le Y\!/\be S)}
\def\nbrxys{{\nbr}(\lbe X\!\times_{\be S}\!Y\be/\lbe S)}

\def\bxs{\brp(\lbe X\!/\be S\le)}
\def\bys{\brp(\le Y/S\e)}

\usepackage[usernames,dvipsnames]{xcolor}
\usepackage{tabularx}
\usepackage{epsfig,amssymb}
\usepackage{bm} 
\usepackage{verbatim} 

\usepackage{hyperref}

\definecolor{labelkey}{rgb}{1,0,0}

\usepackage{verbatim} 

\makeatletter

\DeclareMathAlphabet{\mathcalligra}{T1}{calligra}{m}{n}

\numberwithin{equation}{section}

\newcommand{\sh}{\kern -.4em\phantom{a}^{\mathbf{\sim}}}

\newcommand{\lra}{\longrightarrow}

\newcommand{\et}{{\rm {\acute et}}}

\newcommand{\fppf}{{\rm fppf}}

\newcommand{\ks}{k^{\e\rm s}}

\def\be{\kern -.1em}
\def\le{\kern 0.03em}
\def\lle{\kern 0.015em}
\def\lbe{\kern -.025em}

\newcommand{\Z}{{\mathbb Z}}

\newcommand{\N}{{\mathbb N}}

\newcommand{\spec}{\mathrm{ Spec}\,}

\newcommand{\krn}{\mathrm{Ker}\e}
\newcommand{\img}{\mathrm{Im}\e}
\newcommand{\cok}{\mathrm{Coker}\,}

\def\e{\kern 0.08em}

\newcommand{\xs}{X^{\rm{s}}}
\newcommand{\ys}{Y^{\lle\rm{s}}}

\newtheorem{lemma}{Lemma}[section]
\newtheorem{theorem}[lemma]{Theorem}
\newtheorem{proposition-definition}[lemma]{Proposition-Definition}
\newtheorem{corollary}[lemma]{Corollary}
\newtheorem{proposition}[lemma]{Proposition}
\theoremstyle{definition}
\newtheorem{definition}[lemma]{Definition}

\theoremstyle{remark}
\newtheorem{remark}[lemma]{Remark}
\newtheorem{remarks}[lemma]{Remarks}

\definecolor{labelkey}{rgb}{1,0,0}

\begin{document}

\input xy     
\xyoption{all}

\title[The units-Picard complex and the Brauer group]{The relative units-Picard complex and the Brauer group of a product}

\author{Cristian D. Gonz\'alez-Avil\'es}
\address{Departamento de Matem\'aticas, Universidad de La Serena, Cisternas 1200, La Serena 1700000, Chile}
\email{cgonzalez@userena.cl}
\thanks{Partially supported by Fondecyt grant 1160004.}
\date{\today}


\maketitle

\topmargin -1cm

\smallskip

\maketitle

\topmargin -1cm

\begin{abstract} We introduce the relative units-Picard complex of an arbitrary morphism of schemes and apply it to the problem of describing the (cohomological) Brauer group of a (fiber) product of schemes in terms of the Brauer groups of the factors. Under certain hypotheses, we obtain a five-term exact sequence involving the preceding groups which enables us to solve the indicated problem for, e.g., a class of ruled varieties over a field of characteristic zero. 
\end{abstract}

\section{Introduction}

In this paper a {\it $k$-variety}, where $k$ is a field, is a geometrically integral $k$-scheme of finite type.

Let $k$ be a field with fixed separable closure $\ks$, set $\g={\rm Gal}(\ks\be/k)$ and let $X$ be a $k$-variety. When ${\rm char}\, k=0$, Borovoi and van Hamel introduced in \cite{bvh} a complex ${\rm UPic}\le(\xs)$ of $\g$-modules of length 2 (in degrees $0$ and $1$) whose zero-th cohomology is the group of relative units $U_{k}(\xs)=\ks[X]^{*}\be/(\ks)^{*}$ and first cohomology is  $\pic\xs$, where $\xs=X\times_{k}\ks$. As an application of their construction, they showed that a certain variant of the elementary obstruction of Colliot-Th\'el\`ene and Sansuc \cite[Definition 2.2.1, p.~413]{cts} to the existence of a rational point on a torsor under a (smooth) linear algebraic $k$-group $G$ corresponds to a class in Borovoi's abelian cohomology group $H^{1}_{\rm ab}(k,G\e)$, thereby providing a cohomological interpretation of this obstruction. Later, Harari and Skorobogatov, motivated also by their investigations of rational (and integral) points on schemes over number fields (and their rings of integers), introduced in \cite{hsk} a generalization of the Borovoi-van Hamel complex and used it to extend to a general smooth $k$-variety the descent theory of Colliot-Th\'el\`ene and Sansuc \cite{cts}, which is developed under the hypothesis $U_{k}(\lbe\xs)=0$. The aforementioned descent theory depends, in part, on rather intricate computations with spectral sequences (see, especially, the 11-page appendix \cite[Appendix 1.A, pp.~397-407]{cts}), as well as on some delicate explicit computations with cocycles \cite[Proposition 3.3.2 and Lemma 3.3.3]{cts}. Harari and Skorobogatov showed that, by working instead (primarily) with derived categories and hypercohomology of complexes, a number of technical simplifications could be achieved in their (more general) descent theory.

In this paper we associate to an arbitrary morphism of schemes $f\colon X\to S$ a relative units-Picard complex $\upicxs$ in degrees $-1$ and $0$ (regarded as an object of the derived category $\dbs$) which generalizes the Borovoi-van Hamel and Harari-Skorobogatov complexes. If $f$ is schematically dominant (as is the case when $S=\spec k$, where $k$ is a field), $H^{-1}\lbe(\e\upicxs)=\uxs=\cok\be[\e\bg_{m,\e S}\to f_{*}\bg_{m,\e X}]$ is the \'etale sheaf of relative units of $X$ over $S$, where the indicated morphism of \'etale sheaves on $S$ is induced by $f$, and $H^{\le 0}\lbe(\e\upicxs)=\picxs=R^{\le 1}_{\et}\le f_{\lbe *}\bg_{m,X}$ is the (\'etale) Picard functor of $X$ over $S$. Thus $\upicxs$ is, indeed, a natural generalization of the Borovoi-van Hamel complex ${\rm UPic}\le(\xs)$ and, like ${\rm UPic}(\xs)$, $\upicxs$ is well-suited for studying obstructions to the existence of a section of $f$ when $X$ is a torsor under a smooth $S$-group scheme $G$ with connected fibers whose maximal fibers $G_{\lbe\eta}$ are {\it rational} over $k(\eta)^{\rm s}$. The indicated rationality hypothesis arises as follows. When ${\rm char}\, k=0$, a key ingredient for relating the Borovoi-van Hamel elementary obstruction to the abelian cohomology group $H^{1}_{\rm ab}(k,G\e)$, via \cite[Lemma 6.4]{san}, is the additivity lemma \cite[Lemma 5.1]{bvh}
\begin{equation}\label{upicadd-0}
{\rm UPic}\le(\xs\!\times_{\ks}\! \ys\le)\simeq {\rm UPic}(\xs)\oplus {\rm UPic}(\ys),
\end{equation}
where $X$ and $Y$ are smooth $k$-varieties and $\ys$ is rational. The proof of \eqref{upicadd-0} combines Rosenlicht's additivity theorem
$U_{k}(\xs)\oplus\e U_{k}(\ys)\isoto U_{k}(\e\xs\be\times_{\ks}\be \ys)$ \cite[Theorem 2]{ros} and a lemma of Colliot-Th\'el\`ene and Sansuc \cite[Lemma 11, p.~188]{cts77} which states that the canonical injection $\pic \xs\lbe\oplus\lbe \pic \ys\hookrightarrow \pic\be(\xs\!\times_{\ks}\!\ys)$ is an isomorphism if $X$ and $Y$ are smooth $k$-varieties and $\ys$ is rational. Rosenlicht's additivity theorem holds true in a much broader setting and independently of any rationality hypotheses. Indeed, it is shown in \cite[Corollary 4.3]{ga} that, if $S$ is a reduced scheme and $f\colon X\to S$ and $g\colon Y\to S$ are faithfully flat morphisms locally of finite presentation with reduced and connected maximal geometric fibers and $f\!\times_{\lbe S}\lbe g\e\colon X\be\times_{S}\be Y\to S$ has an \'etale quasi-section (e.g., is smooth), then there exists a canonical isomorphism of \'etale sheaves on $S$
\begin{equation}\label{uadd}
\uxs\lbe\oplus\lbe\uys\isoto U_{X\times_{S}Y/S}.
\end{equation}
On the other hand, that part of the proof of \eqref{upicadd-0} in \cite{bvh} which relies on the Colliot-Th\'el\`ene-Sansuc isomorphism can be shown to follow from the weaker statement
\begin{equation}\label{app}
\pic\be(\xs\!\times_{k^{\e\rm s}}\!\ys\le)^{\g}\simeq (\e\pic\xs)^{\g}\be\oplus\be(\e\pic\ys)^{\g},
\end{equation}
which certainly holds if either $\xs$ or $\ys$ is rational but also, for example, for pairs of abelian varieties $X$ and $Y$ such that $\Hom_{\le k}\lbe(X,Y)=0$ \cite[Corollary 5.11]{ga}. When \eqref{app} is satisfied by the maximal fibers of $f\times_{S}g\colon X\times_{S}Y\to S$, and certain additional hypotheses on $S, X, Y, f$ and $g$ hold, the canonical morphism $\picxs\oplus\picys\to\picxys$ is an isomorphism of \'etale sheaves on $S$. In conjunction with \eqref{uadd}, the latter yields an isomorphism
\begin{equation}\label{upadd-0}
\upicxys\simeq\upicxs\oplus\upicys
\end{equation}
in $\dbs$ which generalizes \eqref{upicadd-0}. In analogy to the Borovoi-van Hamel approach, \eqref{upadd-0} will be shown (elsewhere) to be relevant in relating certain obstructions to the existence of a section of $X\to S$, where $S$ is a connected Dedekind scheme and $X$ is a torsor under a reductive $S$-group scheme $G$, to the {\it abelian class group} $C_{\rm ab}(G\e)$ of $G$ introduced in \cite{ga3}, which is a certain subgroup of  $H^{1}_{\rm ab}(S_{\fppf},G\e)$. 

\smallskip

In the present paper we develop a different application of \eqref{upadd-0}, namely to the problem of describing the (cohomological) Brauer group of a product of schemes in terms of the Brauer groups of the factors.

\smallskip

For any scheme $X$, let $\brp X=H^{\le 2}\lbe(X_{\et},\bg_{m,\le X}\lbe)$ be the 
cohomological Brauer group of $X$ and, for any morphism of schemes $f\colon X\to S$, let $\brxs=R^{\le 2}_{\et}\le f_{\lbe *}\bg_{m,\le X}$. Define the {\it \'etale index $I(\le f\le)=I(X)$ of $f$} as the greatest common divisor of the degrees of all finite \'etale quasi-sections of $f$ of constant degree, if any exist. Note that $I(\le f\le)$ is defined (and is equal to 1) if $f$ has a section. There exist a canonical homomorphism of abelian groups $\brp X\to\brxs(\lbe S\e)$ whose kernel (respectively, cokernel) is denoted by $\bro$ (respectively, $\brt$) and a canonical map $\brp S\to\bro$ whose cokernel is denoted by $\bra$. When $S=\spec k$, where $k$ is a field, and $X$ is a $k$-variety, the groups ${\rm{Br}}_{1}^{\le\prime}\lbe(\lbe X\be/\le k\le)$ and ${\rm{Br}}_{\lbe\rm{a}}^{\le\prime}\lbe(\lbe X\be/\le k\le)$ were discussed by Sansuc in \cite{san}. If $X$ is smooth and projective, the group $\br_{\be 2}^{\le\prime}(\be X\!/k)$ was discussed by Colliot-Th\'el\`ene and Skorobogatov in \cite{ctsk} (and denoted by $\cok\alpha$ there). By \cite[Theorem 2.1]{ctsk}, $\br_{\be 2}^{\le\prime}(\be X\!/k)$ is a finite group if $k$ is a field of characteristic zero.

\smallskip

Now, it is shown in \cite[Lemma 6.3, (ii) and (iv)]{san} that, if either $X(k)\neq\emptyset$ or $H^{\le 3}(k,\bg_{m,\le k})=0$, then ${\rm{Br}}_{\lbe\rm{a}}^{\le\prime}\lbe(\lbe X\be/\le k\le)$ fits into an exact sequence of abelian groups
\begin{equation}\label{ss}
\dots\to H^{2}(k,U_{k}(\xs\e))\to {\rm{Br}}_{\lbe\rm{a}}^{\le\prime}\lbe(\lbe X\be/\le k\le)\to
H^{1}(k,\pic\xs\e)\to\dots.
\end{equation}
Sansuc derived the above sequence from the exact sequence of terms of low degree corresponding to the Hochschild-Serre spectral sequence $H^{\le p}(\g,H^{\le q}(\xs,\bg_{m,\le \xs}))\implies H^{p+q}(X,\bg_{m,\le X})$, where $\g={\rm Gal}\big(\e\ks\be/k\le\big)$. Under the same hypotheses, Borovoi and van Hamel showed  that there exists a canonical isomorphism of abelian groups
\begin{equation}\label{iss}
{\rm{Br}}_{\lbe\rm{a}}^{\le\prime}\lbe(\lbe X\be/\le k\le)\isoto H^{1}(k,{\rm UPic}(\e\xs\e))
\end{equation}
\cite[Corollary 2.20(ii)]{bvh} which identifies \eqref{ss} with a part of an infinite hypercohomology sequence arising from a distinguished triangle in the derived category $D^{b}((\spec k)_{\et})$. The isomorphism \eqref{iss} is, in fact, valid over any base scheme $S$ in the following sense: if either $I(\le f\le)$ is defined and is equal to $1$ or $H^{\le 3}\lbe(S_{\et},\bg_{m,S})=0$, then there exists a canonical isomorphism of abelian groups
\begin{equation}\label{iss2}
\bra\isoto H^{\le 1}\lbe(S_{\et},\upicxs).
\end{equation}
Under an appropriate set of conditions (see Corollary \eqref{bradd}), \eqref{upadd-0} and \eqref{iss2} together yield the additivity theorem
\begin{equation}\label{uno}
\braxy\simeq\bra\be\oplus\be\bray,
\end{equation}
which generalizes \cite[Lemma 6.6(ii)]{san}.

Under another appropriate set of conditions, we show in the proof of Proposition \ref{bro!} that the canonical map
\begin{equation}\label{bro2}
\brt\be\oplus\be \brty\to\brtxy
\end{equation}
is injective. We then appropriately combine \eqref{uno} and \eqref{bro2} to obtain the following statement, which is the main theorem of the paper:

\begin{theorem}\label{3} {\rm (= Theorem \ref{main})} Let $S$ be a locally noetherian regular scheme and let $f\colon X\to S$ and $g\colon Y\to S$ be smooth and surjective morphisms. Assume that	
\begin{enumerate}
\item[(i)] the \'etale index $I\lbe(\le f\!\times_{\be S}\be g\le)$ is defined and is equal to $1$, 
\item[(ii)] for every point $s\in S$ of codimension $\leq 1$, the fibers $X_{\lbe s}$ and $Y_{\lbe s}$ are geometrically integral, and
\item[(iii)] for every maximal point $\eta$ of $S$,
\[
\pic\be(X_{\lbe\eta}^{\le\rm s}\be\times_{k(\eta)^{\rm s}}\be Y_{\!\eta}^{\e\rm s}\le)^{\g(\eta)}=(\e\pic X_{\lbe\eta}^{\le\rm s})^{\g(\eta)}\be\oplus\be(\e\pic Y_{\!\eta}^{\e\rm s})^{\g(\eta)},
\]
where $\g(\eta)={\rm Gal}(k(\eta)^{\rm s}/k(\eta))$.
\end{enumerate}	
Then there exists a canonical exact sequence of abelian groups
\[
\begin{array}{rcl}
0\to\brp S\to\brp X\!\oplus\!\brp\, Y\to \brp(X\!\be\times_{\be S}\be\lbe Y\le)&\to&\displaystyle\frac{\brxys\lbe(\lbe S\e)}{\brxs(\lbe S\e)\!\oplus\! \brys(\lbe S\e)}\\\\
&\to& \displaystyle\frac{\brtxy}{\brt\!\oplus\!\brty}\to 0.
\end{array}
\]
\end{theorem}

We now state some consequences of the above theorem for smooth and projective $k$-varieties over a field $k$ (chiefly of characteristic zero).

\smallskip

As noted above, hypothesis (iii) of the theorem is satisfied by pairs of abelian $k$-varieties $A$ and $B$ such that $\Hom_{\e k}(A,B)=0$. Thus the theorem immediately yields the following statement, which is a particular case of Corollary \ref{av}.

\begin{corollary}\label{4} Let $k$ be a field and let $A$ and $B$ be abelian varieties over $k$ such that $\Hom_{\le\lle k}(A,B\le)=0$. Then there exists a canonical exact sequence of abelian groups
\[
\begin{array}{rcl}
0\to\br\le k\to\be\br A\lbe\oplus\lbe\br B\to \br\lbe(A\!\times_{\lbe k}\be B\le)&\to& \displaystyle\frac{\br\lbe(\lbe A^{\le\rm s}\be\times_{\lbe \ks}\! B^{\e\rm s})^{\g}}{(\br\lbe A^{\le\rm s})^{\g}\be\!\oplus\! (\br B^{\e\rm s})^{\g}}\\\\
&\to&\displaystyle\frac{\br_{\be 2}(A\be\times_{k}\be B\be/k)}{\br_{\be 2}( A/k)\!\oplus\!\br_{\be 2}(\le B\lbe/k)}\to 0.
\end{array}
\]
\end{corollary}

When $k$ is a field of characteristic 0 and one of the factors is a smooth and projective $k$-variety $X$ such that $H^{1}(\xs,\s O_{\be X^{\lbe\rm s}})=0$, then Theorem \ref{3} yields particularly simple statements. For example:

\begin{corollary}\label{niI} {\rm (=part of Theorem \ref{ni})} Let $k$ be a field of characteristic zero and let $X$ and $\e Y$ be smooth and projective $k$-varieties. Assume that
\begin{itemize}
\item[(i)] the geometric N\'eron-Severi group ${\rm NS}\!\lbe\left(\e\xs\e\right)$ is torsion-free,
\item[(ii)] either $H^{1}(\xs,\s O_{\be X^{\lbe\rm s}})=0$ or $H^{1}(\ys,\s O_{ Y^{\lle\rm s}})=0$ and
\item[(iii)] the index of $X\!\be\times_{k}\! Y$ is $1$.
\end{itemize}	
Then there exists a canonical exact sequence of abelian groups
\[
0\to\br k\to\br\lbe\lbe(X\!\be\times_{\lbe k}\!\lbe Y\le)\to (\le\br X\be/\e\br\le k\lle)\!\oplus\!(\le\br Y\be/\e\br\le k\lle)\to 0.
\]
If $(X\!\be\times_{k}\! Y\le)(k)\neq\emptyset$, then the choice of a $k$-rational point on $X\!\be\times_{k}\! Y$ determines a splitting of the above sequence.
\end{corollary}

The proof of the corollary relies on a computation of Skorobogatov and Zarhin \cite{sz} which depends on hypothesis (i) above. The next statement is immediate from the above corollary.

\begin{corollary}\label{mkor} Let $k$ be a field of characteristic zero such that $\br k=0$ and let $X$ and $\e Y$ be smooth and projective $k$-varieties such that hypotheses {\rm (i)-(iii)} of Corollary {\rm \ref{niI}} hold. Then there exists a canonical isomorphism of abelian groups
\[
\br\lbe\lbe( X\!\be\times_{\lbe k}\!\lbe Y\lle)=\br X\!\oplus\be\br Y.
\]
\end{corollary}

Another consequence of Corollary \ref{niI} is the following statement, which applies, in particular, to certain types of smooth and projective $k$-ruled $k$-varieties, i.e., smooth and projective $k$-varieties which are $k$-birationally equivalent to $X\be\times_{k}\be \mathbb P^{ \le 1}_{\!\be k}$ for some smooth and projective $k$-variety $X$.

\begin{corollary} \label{ni2} Let $k$ be a field of characteristic zero and let $Z$ be a smooth and projective $k$-variety which is $k$-birationally equivalent to a product $X\be\times_{k}\be Y$ of smooth and projective $k$-varieties, where $H^{1}(\e\ys\be,\s O_{ Y^{\lle\rm s}})=0$ and $\br\le Y=\br\le k$.
If ${\rm NS}\!\lbe\left(\le\xs\e\right)$ is torsion-free and $X\!\be\times_{k}\! Y$ has index $1$, then there exists an isomorphism of abelian groups $\br Z\simeq \br X$.
\end{corollary}
\begin{proof} By the birational invariance of the Brauer group of proper and regular schemes over a field of characteristic zero, it suffices to check that the canonical pullback map $\br X\to \br\lbe\lbe(X\be\times_{\lbe k}\lbe Y\le)$ is an isomorphism of abelian groups. This follows from the commutative diagram of abelian groups
\[
\xymatrix{0\ar[r]& \br k\ar[r]\ar@{=}[d]&\br X\ar[r]\ar[d]& \br X\!/\e\br\le k\ar[r]\ar@{=}[d]& 0\\
0\ar[r]& \br k\ar[r]&\br\lbe\lbe(X\!\be\times_{\lbe k}\!\lbe Y\le)\ar[r]& \br X\!/\e\br\le k\ar[r]& 0,
}
\]
whose top (respectively, bottom) row is exact since $X$ has index 1 (respectively, by Corollary \ref{niI}).
\end{proof}

\section*{Acknowledgements}
I thank the referee for constructive criticism and for suggesting that my initial results should apply, more generally, to certain pairs of smooth and projective $k$-varieties where neither of the factors is $\ks$-rational.  I also thank Stefan Gille for answering my questions and for his encouragement, David Harari for valuable advice and Dino Lorenzini for a suggestion that led me to Corollary \ref{ni2}. 

\section{Preliminaries}

If $A$ is an object of a category, the identity morphism of $A$ will be denoted by $1_{\be A}$. The category of abelian groups will be denoted by $\mathbf{Ab}$.

 When $f\colon A\to B$ is an injective homomorphism of abelian groups (so that $A$ can be identified with $\img f$), we will sometimes write $B/A$ for $\cok f$ (this will be done chiefly to avoid introducing additional notation).

\smallskip

The following consequence of the snake lemma will be applied several times in this paper.

\begin{lemma}\label{ker-cok} If $\mathcal A$ is an abelian category and $f$ and $g$ are morphisms in $\mathcal A$ such that $g\be\circ\!\be f$ is defined, then there exists a canonical exact sequence in $\mathcal A$ 
\[
0\to\krn f\to\krn\lbe(\e g\be\circ\!\be f\e)\to\krn g\to\cok f\to\cok\be(\e g\be\circ\!\be f\e)\to\cok g\to 0.
\]
\end{lemma}

If $\mathcal A$ is an abelian category, we will write $C^{\e b}\lbe(\mathcal A)$ for the category of bounded complexes of objects of $\mathcal A$. The corresponding derived category will be denoted by $D^{\e b}\lbe(\mathcal A)$.

Recall that the {\it mapping cone} of a morphism of complexes $u\colon A^{\bullet}\to B^{\le\bullet}$ in $C^{\e b}\lbe(\mathcal A)$ is the complex $C^{\le\bullet}(u)$ whose $n$-th component, where $n\in\Z$, is $C^{\e n}(u)=A^{n+1}\oplus B^{\le n}$ with differential $d_{\lle C\lbe(u)}^{\e n}\lbe(a,b)=(-d^{\e n+1}_{\be A}\lbe(a), u(a)+d_{\lbe B}^{\e n}(b))$, where $a\in A^{n+1}$ and $b\in B^{\le n}$. Note that, if $u\colon A\to B$ is an injective morphism in $\mathcal A$, then $C^{\le\bullet}(u)=\cok u$ in $D^{\e b}\lbe(\mathcal A)$. Further, if 
\[
0\to A^{\le\bullet}\overset{\!u}{\to} B^{\le\bullet}\overset{\!v}{\to}  C^{\le\bullet}\to 0
\]
is an exact sequence in $C^{\e b}\lbe(\mathcal A)$, then $v$ induces a quasi-isomorphism $C^{\le\bullet}(u)\simeq C^{\le\bullet}$ and we obtain a distinguished triangle in $D^{\e b}\lbe(\mathcal A)$:
\[
A^{\le\bullet}\to B^{\le\bullet}\to C^{\le\bullet}\to A^{\le\bullet}[1].
\]
See, for example, \cite[p.~20]{lip}.

Now, if $A^{\bullet}$ is a complex in $C^{\e b}\lbe(\mathcal A)$ and $n\in\Z$, then the {\it $n$-th truncation} of $A^{\bullet}$ is the complex 
\[
\tau_{\leq\le n}\le A^{\bullet}=\dots \to A^{n-2}\to A^{n-1}\to\krn\lbe[\le A^{n}\to A^{n+1}]\to 0.
\]
For every $n\in\Z$, there exists a canonical exact sequence in $C^{\e b}\lbe(\mathcal A)$
\[
0\to \tau_{\leq\e  n-1}A^{\bullet}\to \tau_{\leq\e  n}\le A^{\bullet}\to H^{\le n}\be(A^{\bullet})[-n]\to 0
\]
which induces a distinguished triangle in $D^{\e b}(\mathcal{A})$:
\begin{equation}\label{n}
\tau_{\leq\e  n-1}A^{\bullet}\to \tau_{\leq\e  n}\le A^{\bullet}\to H^{n}(A^{\bullet})[-n]\to(\tau_{\leq\e  n-1}A^{\bullet})[1].
\end{equation}

\begin{lemma}\label{kc-c} Let $\mathcal A$ be an abelian category and let $A^{\e\bullet}\overset{\!u}{\to}B^{\e\bullet}\overset{\!v}{\to}C^{\e\bullet}$ be  morphisms in $C^{\e b}(\mathcal A)$. Then there exists a distinguished triangle in $D^{\e b}(\mathcal{A})$
\[
C^{\le\bullet}(u)\to C^{\le\bullet}(\le v\be\circ\be u\le)\to C^{\le\bullet}(v)\to C^{\le\bullet}(u)[1].
\]
\end{lemma}
\begin{proof} This follows by applying the octahedral axiom \cite[TRIV, p.~94]{v} to the  distinguished triangles 
\[
\begin{array}{rcl}
& A^{\e\bullet}\overset{\!u}{\to} B^{\e\bullet}\to C^{\e\bullet}(u)\to A^{\e\bullet}[1]\\
& B^{\e\bullet}\overset{\!v}{\to} C^{\e\bullet}\to C^{\e\bullet}(v)\to B^{\e\bullet}[1]\\
& A^{\e\bullet}\!\overset{\!v\e\circ\lle u}{\lra}\! C^{\e\bullet}\to C^{\le\bullet}(\le v\lbe\circ\lbe u\le)\to A^{\e\bullet}[1]
\end{array}.
\]
\end{proof}

If $S$ is a scheme, we will write $S_{\et}$ for the small \'etale site over $S$. The category of abelian sheaves on $S_{\et}$ will be denoted by $S_{\et}^{\le\sim}$. If $F\in S_{\et}^{\le\sim}$ and $g\colon T\to S$ is an \'etale morphism, then the inverse image sheaf $g^{*}F\in T_{\et}^{\e\sim}$ is given by
\begin{equation}\label{inv}
(\e g^{*}\be F\e)(Z\le)=F(Z\le),
\end{equation}
where $Z\to T$ is an \'etale morphism and, on the right-hand side of the above equality, $Z$ is regarded as an $S$-scheme via the composite morphism $Z\to T\overset{\!g}{\to} S$. See, e.g., \cite[p.~89]{t}. If $G$ is an $S$-group scheme, then $G(Z\le)=G_{T}(Z\le)$ and \eqref{inv} yields the following equality in $T_{\et}^{\e\sim}$: 
\begin{equation}\label{inv2}
g^{*}\lbe G=G_{T}.
\end{equation}

Recall that a morphism of schemes $f\colon X\to S$ is called {\it schematically dominant} if the canonical homomorphism $f^{\#}\colon \s O_{\be S}\to f_{*}\s O_{\be X}$ of Zariski sheaves on $S\e$ is injective \cite[\S5.4]{ega1}. For example, if $S$ is reduced and $f\colon X\to S$ is dominant, then $f$ is schematically dominant by \cite[Proposition 5.4.3, p.~284]{ega1}. Further, if $f\colon X\to S$ has a section, then $f$ is schematically dominant by \cite[Lemma 2.3(iii)]{ga}.

\smallskip

Recall now that an {\it \'etale quasi-section} of a morphism $f\colon X\to S$ is an \'etale and surjective morphism $\alpha\colon T\to S$ such that there exists an $S$-morphism $h\colon T\to X$, i.e., $f\circ h=\alpha$. Note that the induced morphism $(h,1_{T})_{S}\colon T\to X_{T}$ is then a section of $f_{T}\colon X_{T}\to T$. By \cite[${\rm IV}_{4}$, Corollary 17.16.3(ii)]{ega}, every smooth and surjective morphism of schemes $X\to S$ has an \'etale quasi-section.

\begin{definition}\label{eti} Let $f\colon X\to S$ be a morphism of schemes which has a finite \'etale quasi-section of constant degree. Then the {\it \'etale index} $I(\le f\le)$ of $f$ is the greatest common divisor of the degrees of all finite \'etale quasi-sections of $f$ of constant degree.
\end{definition}

If $f\colon X\to S$ has a finite \'etale quasi-section $T\to S$ of constant degree $d$ (so that $I(\le f\le)$ is defined) and $Y\to S$ is any $S$-scheme, then $Y_{T}\to Y$ is a finite \'etale quasi-section of $f_{Y}$ of constant degree $d$. It follows that $I(\le f_{Y}\le)$ is also defined and
\begin{equation}\label{genf-0}
I(\le f_{Y}\be)\!\mid\! I(\le f\le).
\end{equation}
In particular, for every maximal point $\eta$ of $S$, the \'etale (i.e., separable) index of $f_{\eta}\colon X_{\eta}\to\eta$ divides $I(\le f\le)$, i.e., $I(\le f_{\eta}\le)\!\mid\! I(\le f\le)$. Note also that, if $g\colon Y\to S$ is another morphism of schemes and $\alpha\colon T\to S$ is an \'etale quasi-section of $f\!\times_{\be S}\be g\colon X\!\times_{\be S}\! Y\to S$ of constant degree $d$ with associated $S$-morphism $h\colon T\to X\!\times_{\be S}\be Y$, then $\alpha$ is also an \'etale quasi-section of $f$ and $g$ with associated $S$-morphisms $p_{X}\be\circ\be h\colon T\to X$ and $p_{\e Y}\be\circ\be h\colon T\to Y$, respectively. It follows that, if $I\lbe (\le f\be\times_{\be S}\be g\le)$ is defined, then $I\le(\le f\le)$ and $I\le(\e g\le)$ are both defined and they divide $I\lbe (\le f\!\times_{\be S}\be g\le)$. Consequently
\begin{equation} \label{genf}
I\lbe (\le f\!\times_{\be S}\be g\le)=1\implies I(\le f_{\eta}\le)=I(\le g_{\eta}\le)=1
\end{equation}
for every maximal point $\eta$ of $S$.

Note, in addition, that the degree of a finite \'etale quasi-section $\alpha\colon T\to S$ of $f\colon X\to S$ is constant over each connected component of $S$. Thus, if $S$ is connected, then the degree of $\alpha$ is constant. Henceforth, the statement {\it $f$ has a finite \'etale quasi-section of constant degree} will be abbreviated to {\it the \'etale index of $f$ is defined}. Note that the statement {\it $I(\le f\le)$ is defined and is equal to $1$} clearly holds if $f$ has a section. 

\smallskip

If $k$ is a field, we let $\ks$ be a fixed separable closure of $k$ and write $\g=\mathrm{Gal}\le(\ks\be/k)$ for the corresponding absolute Galois group. If $X$ is a $k$-scheme, we will write $\xs=X\be\times_{k}\le\spec\ks$.

\begin{remarks}\indent
\begin{enumerate}
\item[(a)] If $k$ is a field and $X\to\spec k$ is a (non-empty) geometrically reduced $k$-scheme locally of finite type, then the \'etale (i.e., separable) index $I(X)$ of $X$ is defined. See \cite[p.~6]{ga}.
\item[(b)] If $S$ is a non-empty open affine subscheme of a proper, smooth and geometrically connected curve over a finite field and $f\colon X\to S$ is a smooth and surjective morphism with geometrically irreducible generic fiber, then $f$ has a finite \'etale quasi-section of constant degree (since $S$ is connected), i.e., $I(\le f\le)$ is defined. See \cite[Theorem (0.1)]{tam}.
\end{enumerate}
\end{remarks}

If $f\colon X\to S$ is a morphism of schemes, we will write
\begin{equation}\label{fb}
f^{\e\flat}\colon \bg_{m,\le S}\to f_{\lbe *}\bg_{m,\le X}
\end{equation}
for the canonical morphism of abelian sheaves on $S_{\et}$ induced by $f$. Let $Y\overset{\!g}{\to} X\overset{\!f}{\to}S$ be morphisms of $S$-schemes. We will make the identification $(\e f\circ g)_{\lbe *}=f_{\le *}\be\circ g_{\le *}$. Then $(\e f\circ g)^{\le\flat}\colon \bg_{m,\le S}\to f_{\lbe *}(\e g_{*}\bg_{m,\le Y}\be)$ factors as
\[
\bg_{m,\le S}\overset{f^{\le\flat}}{\to}f_{*}\bg_{m,\le X}\overset{f_{\lbe *}\lbe(\le g^{\le\flat}\lbe)}{\lra}f_{*}\lbe(\e g_{*}\bg_{m,\le Y}\be),
\]
i.e., 
\begin{equation}\label{marv}
f_{*}\lbe(g^{\flat})\be \circ\be  f^{\le\flat}=(\e f\be\circ\be g\le)^{\flat}.
\end{equation}
 In particular, if $\sigma\colon S\to X$ is a section of $f$, then \eqref{marv} yields
\begin{equation}\label{marv2}
f_{\lbe *}\lbe(\sigma^{\le\flat})\be \circ\be  f^{\le\flat}=1_{\le\bg_{m,\lle S}}.
\end{equation}

Recall now that, if $f\colon X\to S$ is  morphism of schemes, then the (\'etale) {\it relative Picard functor of $X$ over $S$} is the \'etale sheaf $\picxs$ on $S$ associated to the presheaf $(\textrm{Sch}/S\e)\to \mathbf{Ab}, (T\to S\le)\mapsto\pic X_{T}$. We have
\begin{equation}\label{picxs}
\picxs=R^{\le 1}_{\et}\le f_{\lbe *}\bg_{m,X}.
\end{equation}
See \cite[\S2]{klei} and/or \cite[\S8.1]{blr} for basic information on $\picxs$.

We will write $\br X$ for the Brauer group of equivalence classes of Azumaya algebras on $X$ and $\brp X$ for the {\it full} cohomological Brauer group of $X\e$, i.e., $\brp X=H^{2}(X_{\et},\bg_{m,\le X}\lbe)$ (see the following Remark). Further, we let
\begin{equation}\label{brxs1}
\brxs=R^{\le 2}_{\et}\le f_{\lbe *}\bg_{m,\le X}
\end{equation}
be the sheaf on $S_{\et}$ associated to the presheaf $(T\to S)\mapsto \brp X_{T}$. 

\begin{remark}\label{bgps} The torsion subgroup $(\brp X)_{\tors}=H^{2}(X_{\et},\bg_{m,\le X}\lbe)_{\rm tors}$ of $\brp X$ is often (also) called the cohomological Brauer group of $X$. If $X$ is regular and noetherian, then $\brp\lbe X$ is, in fact, a torsion group \cite[II, Proposition 1.4]{gb}, whence $\brp X=(\brp X)_{\tors}$. In general, there exists a canonical injection
$\br X\to \brp\lbe X$ which factors through $(\brp X)_{\tors}$ if $X$ is quasi-compact \cite[I, \S2]{gb}.  On the other hand, by a theorem of Gabber \cite{J}, the canonical map $\br X\to (\brp X)_{\tors}$ is an isomorphism if $X$ is quasi-compact, separated and admits an ample invertible sheaf. We conclude that, if $X$ is noetherian, regular, separated and there exists 
an ample invertible sheaf on $X$, then $\br X=(\brp X)_{\tors}=\brp X$.
\end{remark}

The Cartan-Leray spectral sequence associated to $f$
\begin{equation}\label{leray}
H^{\le r}\be(S_{\et}, R^{\e s}\! f_{\lbe *}\bg_{m,X})\Rightarrow H^{\le r+s}\lbe(X_{\et}, \bg_{m,X})
\end{equation}
provides edge morphisms $e_{\be f}^{r}\colon H^{\le r}\lbe(S_{\et},f_{\lbe *}\bg_{m,X})\to H^{\le r}\lbe(X_{\et},\bg_{m,X})$ for every $r\geq 0$. On the other hand, there exists a canonical pullback map
\begin{equation}\label{rr}
f^{\le(r)}\colon H^{\le r}\be(S_{\et},\bg_{m,S})\to H^{\le r}\be(X_{\et},\bg_{m,X}),
\end{equation}
namely the composition
\[
H^{\le r}\be(S_{\et},\bg_{m,S})\overset{H^{\le r}\be(\le f^{\le\flat})}{\lra}H^{\le r}\be(S_{\et},f_{*}\bg_{m,X})
\overset{e_{\be f}^{r}}{\lra}H^{\le r}\be(X_{\et},\bg_{m,X}),
\]
where $f^{\le\flat}$ is the map \eqref{fb}. When $r=1$ (respectively, $r=2$), \eqref{rr} will be identified with (respectively, denoted by) the canonical map $\pic\be f\colon \pic S\to \pic X$ (respectively, $\brp f\colon \brp S\to \brp X\le$). Note that $f^{(r)}$ \eqref{rr} is injective if the \'etale index $I\lbe(\lle f\lle)$ of $f$ is defined and is equal to 1 \cite[Remark 3.1(d)]{ga}.

The {\it relative cohomological Brauer group of $X$ over $S$} is
\begin{equation}\label{relb}
\bxs=\krn\!\left[\e \brp f\colon \brp S\to \brp X\le\right].
\end{equation}
Thus, we have a functor
\begin{equation}\label{kbr}
\brp(-/S)\colon ({\rm{Sch}}/S\e)\to \mathbf{Ab},\, (\le X\overset{\!f}{\to}S\le)\mapsto \bxs.
\end{equation}
We will also need to consider the abelian groups
\begin{equation}\label{npic}
\npicxs=\cok\pic\be f
\end{equation}
and
\begin{equation}\label{nbr}
\nbrxs=\cok\brp f,
\end{equation}
which define functors
\begin{equation}\label{npicf}
\npic\be(-/\be S)\colon ({\rm{Sch}}/S\e)\to \mathbf{Ab},\, (\le X\to S\le)\mapsto \npic\be(X\be/\be S\le),
\end{equation}
and
\begin{equation}\label{nbrf}
\nbr(-/\be S)\colon ({\rm{Sch}}/S\e)\to \mathbf{Ab},\, (\le X\to S\le)\mapsto \nbrxs.
\end{equation}

Next, there exists a canonical homomorphism of abelian groups
\begin{equation}\label{canb}
\brp X\!\to\be\brxs(S\le)
\end{equation}
which is an instance of the canonical adjoint homomorphism $P(S\le)\to P^{\#}\be(S\le)$, where $P$ is a presheaf of abelian groups on $(\textrm{Sch}/S\le)$ and $P^{\le\#}$ is its associated (\'etale) sheaf \cite[Remark, p.~46]{t}. Let
\begin{equation}\label{br1}
\bro=\krn\!\be\left[\brp X\be\to\be\brxs(S\e)\e\right]
\end{equation}
and
\begin{equation}\label{br2}
\hspace{.5cm}\brt=\cok\be[\e\brp X\to\brxs(\lbe S\e)\e],
\end{equation}
where the indicated map is \eqref{canb}.

Now, by \cite[p.~309, line 8]{mil}, the Cartan-Leray spectral sequence \eqref{leray} induces an exact sequence of abelian groups
\[
\begin{array}{rcl}
0&\to & H^{\le 1}\lbe(S_{\et},f_{\lbe *}\bg_{m,X})\overset{\!e_{\be f}^{1}}{\to}\pic X\to\picxs(S\le)\to
H^{\le 2}\lbe(S_{\et},f_{\lbe *}\bg_{m,X})\\
&\overset{\!\be\lambda_{\lbe X\!\lbe/\be S}}{\to} & \bro\to H^{\le 1}\lbe(S_{\et},\picxs)\to H^{\le 3}\lbe(S_{\et},f_{\lbe *}\bg_{m,X}),
\end{array}
\]
where the map $\lambda_{\lbe X\!/\lbe S}$ has the following property: the composition
\[
H^{\le 2}(S_{\et},f_{\lbe *}\bg_{m,X})\overset{\!\be\lambda_{\lbe X\!\lbe/\be S}}{\to}\bro\hookrightarrow \brp X
\]
is the second edge morphism $e_{\be f}^{2}\colon H^{\le 2}(S_{\et},f_{\lbe *}\bg_{m,X})\to \brp X$. Consequently, the map $\brp\lbe f=e_{\be f}^{2}\circ H^{\le 2}(f^{\le\flat})$ factors as
\[
\brp S\overset{H^{2}\be(\le f^{\le\flat})}{\lra}H^{\le 2}\be(S_{\et},f_{*}\bg_{m,X})\overset{\!\be\lambda_{\lbe X\!\lbe/\be S}}{\to}\bro\hookrightarrow \brp X.
\]
We will write
\begin{equation}\label{brf1}
\br^{\be *}\be f\colon \brp S\to\bro
\end{equation}
for the composition
\[
\brp S\overset{H^{2}\lbe(\le f^{\le\flat})}{\lra}H^{\le 2}\be(S_{\et},f_{*}\bg_{m,X})\overset{\!\be\lambda_{\lbe X\!\lbe/\be S}}{\to}\bro.
\]
Clearly, $\krn \br^{\be *}\be f=\krn \brp\lbe f=\bxs$. Now set
\begin{equation}\label{bra}
\bra=\cok\!\be\left[\e\br^{\be *}\lbe f\colon \brp\lle S\to\bro\e\right].
\end{equation}
Thus, in addition to \eqref{kbr} and \eqref{nbrf}, we have the functors
\begin{equation}\label{br1s}
\br_{\be 1}^{\le\prime}(-/S\le)\colon ({\rm{Sch}}/S\e)\to \mathbf{Ab}, (X\to S)\mapsto\bro,
\end{equation}
\begin{equation}\label{br2s}
\br_{\be 2}^{\le\prime}(-/S\le)\colon ({\rm{Sch}}/S\e)\to \mathbf{Ab}, (X\to S)\mapsto\brt
\end{equation}
and
\begin{equation}\label{braf}
\br_{\be \rm a}^{\le\prime}(-/S\le)\colon ({\rm{Sch}}/S\e)\to \mathbf{Ab}, (X\to S)\mapsto \bra.
\end{equation}
The functors  \eqref{kbr}, \eqref{br1s}, \eqref{br2s} and \eqref{braf} are related by canonical exact sequences
\begin{equation}\label{br-seq}
0\to\bxs\to\brp S\overset{\!\br^{\be *}\!\lbe f}{\lra}\bro\to\bra\to 0
\end{equation}
and
\begin{equation}\label{bps0}
0\to\bro\to\brp X\to\brxs(\lbe S\e)\to \brt\to 0
\end{equation}
where the middle map in the second sequence is the homomorphism \eqref{canb}. An application of Lemma \ref{ker-cok} to the pair of maps
\[
\brp S\overset{\!\br^{\be *}\!\lbe f}{\lra}\bro\hookrightarrow \brp X 
\]
using the exactness of \eqref{br-seq} and \eqref{bps0} yields the following exact sequence which includes $\nbrxs$ \eqref{nbr}:
\begin{equation}\label{bps1}
0\to\bra\to\nbrxs\to\brxs(\lbe S\e)\to \brt\to 0.
\end{equation}

\begin{remarks}\label{int}\indent
\begin{itemize}
\item[(a)] Note that $\br_{(-)/S}$ \eqref{brxs1} defines a functor
$(\textrm{Sch}/S\le)\to S_{\et}^{\le\sim}$ which maps $1_{\lbe S}$ to the zero sheaf. It follows that, if $F\colon ({\rm{Sch}}/S\e)\to \mathbf{Ab}$ denotes one of the functors $\brp(-/S)$, $\nbr(-/\be S)$, $\br_{\be 2}^{\le\prime}(-/S\le)$ or $\br_{\be \rm a}^{\le\prime}(-/S\le)$, then $F(1_{\lbe S})=0$. In particular, all four functors which appear in the sequence \eqref{bps1} have this property. Note, however, that the functor $\br_{\be 1}^{\le\prime}(-/S\le)$ appearing in the sequences \eqref{br-seq} and \eqref{bps0} does not have the indicated property.

\item[(b)] The functor $\br_{\lbe 0}^{\le\prime}\lbe(-/S\le)\colon ({\rm{Sch}}/S\e)\to \mathbf{Ab}, (X\overset{\!f}{\to} S)\mapsto\img\brp\lbe f,$ has also been considered in the literature (at least when $S$ is the spectrum of a field). But this functor will not play an explicit role in this paper.

\item[(c)] If $k$ is a field and $f\colon X\to\spec k$ is quasi-compact and quasi-separated, then $\br_{\!\! X\be/k}^{\le\prime}\lbe(k\e)=(\brp\xs\le)^{\lbe\g}$ by \cite[II, Corollary 2.2(ii), p.~94, formula on p.~118, line 8, and Theorem 6.4.1, p.~128]{t}. In this case the group \eqref{br1} 
\[
\br_{\be\lle 1}^{\le\prime}\lbe(\lbe X\be/k)=\krn\!\be\left[\brp X\be\to\be(\brp\xs\le)^{\lbe\g}\,\right],
\]
is called the (cohomological) {\it algebraic Brauer group of $X$}. The group
\[
\hskip -.3cm\br_{\be{\rm t}}^{\le\prime}\lbe(\lbe X\be/k)=\img\!\be\left[\brp X\be\to\be(\brp\xs\le)^{\lbe\g}\,\right]
\]
is called the (cohomological) {\it transcendental Brauer group of $X$}. Note that $\br_{\be{\rm t}}^{\le\prime}\lbe(\lbe X\!/k)$ is canonically isomorphic to $\krn[\e(\brp\xs\le)^{\lbe\g}\to\br_{\lbe 2}^{\le\prime}(X/k)]$, where the indicated map is the last nontrivial map in \eqref{bps0} for $S=\spec k$.
\item[(d)] If $X$ and $S$ are separated regular and noetherian schemes, both equipped with an ample invertible sheaf, then $\brp S=\br S$ and $\brp X=\br X$ by Remark \ref{bgps}. It then follows that all primed groups defined above (\eqref{relb}, \eqref{nbr}, etc.) coincide with the similarly-defined unprimed groups (of equivalence classes of Azumaya algebras).  
\item[(e)] It is shown in \cite[Theorem 2.1]{ctsk} that, if $k$ is a field of characteristic zero and $X$ is a smooth and projective $k$-variety, then $\br_{\be 2}^{\le\prime}(\be X\!/k)=\br_{\be 2}(\be X\!/k)$ is a finite group.
\end{itemize}	
\end{remarks}

\smallskip

Recall that a locally noetherian scheme $X$ is called {\it locally factorial} if, for every $x\in X$, the local ring $\s O_{\be X,\le x}$ is factorial. The following implications hold for locally noetherian schemes: regular $\implies$ locally factorial $\implies$ normal. Recall also that, if $S$ is a scheme and $\eta$ is a maximal point of $S$, then $\g(\eta)=\mathrm{Gal}\le(k(\eta)^{\rm s}\lbe/\le k)$, where $k(\eta)^{\rm s}$ is a fixed separable closure of $k(\eta)$.

\smallskip

The following statement is a consequence of results in \cite{ga}.

\begin{proposition}\label{opo2} Let $S$ be a locally noetherian normal scheme and let $f\colon X\to S$ and $g\colon Y\to S$ be faithfully flat morphisms locally of finite type. Assume that
\begin{enumerate}
\item[(i)] $X$, $Y$ and $X\be\times_{S}\be Y$ are locally factorial,
\item[(ii)] for every point $s\in S$ of codimension $\leq 1$, the fibers $X_{\lbe s}$ and $Y_{\lbe s}$ are geometrically integral and
\item[(iii)] for every maximal point $\eta$ of $S$,
\begin{enumerate}
\item[(a)] ${\rm gcd}\le(\le I(X_{\eta}\le),I(\le Y_{\eta}\le))=1$ and
\item[(b)] $\pic\be(X_{\lbe\eta}^{\le\rm s}\be\times_{k(\eta)^{\rm s}}\be Y_{\!\eta}^{\e\rm s}\le)^{\g(\eta)}=(\e\pic X_{\lbe\eta}^{\le\rm s})^{\g(\eta)}\be\oplus\be(\e\pic Y_{\!\eta}^{\e\rm s})^{\g(\eta)}$.
\end{enumerate}
\end{enumerate}		
Then the canonical map 
\[
\npic\be(X\be/\be S\le)\oplus\npic\be(Y\be/\be S\le)\to\npic\be(X\!\times_{\be S}\be Y\be/\be S\le)
\]
is an isomorphism of abelian groups.
\end{proposition}
\begin{proof} The general hypotheses together with (i) and (ii) show that the map of the statement can be identified with the canonical map 
\[
\prod \left(\pic X_{\eta}\lbe\oplus\lbe \pic Y_{\eta}\right)\to
\prod\pic\be(X_{\eta}\be\times_{k(\eta)}\be Y_{\eta}\le),
\]
where the products run over the set of maximal points $\eta$ of $S$. See \cite[Proposition 5.4]{ga}. Now, since (iii) holds, the preceding map is an isomorphism by \cite[Proposition 5.8]{ga}.
\end{proof}

\section{The relative units-Picard complex}

Let $f\colon X\to S$ be a morphism of schemes. By \eqref{n}, for every $n\in\N$ there exists a canonical distinguished triangle in $\dbs$
\begin{equation}\label{n1}
\tau_{\leq\e  n-1}\mathbb Rf_{\lbe *}\bg_{m,X}\to\tau_{\leq\e  n}\le\mathbb Rf_{\lbe *}\bg_{m,X}\to R^{\e n}\be f_{\lbe *}\bg_{m,X}[-n]\to(\tau_{\leq\e  n-1}\mathbb Rf_{\lbe *}\bg_{m,X})[1].
\end{equation}
We will write
\begin{equation}\label{ifn}
i_{\be f}^{\le  n}\colon\tau_{\leq\e  n-1}\mathbb Rf_{\lbe *}\bg_{m,X}\to\tau_{\leq\e  n}\le\mathbb Rf_{\lbe *}\bg_{m,X}
\end{equation}
for the first map in \eqref{n1}. Thus there exists a canonical isomorphism in $\dbs$
\begin{equation}\label{coni}
C^{\e\bullet}\lbe(\le i_{\be f}^{n}\le)=R^{\e n}\be f_{\lbe *}\bg_{m,X}[-n].
\end{equation}
If $r\leq n$, we will identify 
$H^{\le r}\be(S_{\et}, \tau_{\leq\,  n}\e\mathbb Rf_{\lbe *}\bg_{m,X})$ and $H^{\le r}(X_{\et},\bg_{m,\le X}\lbe)$ via the canonical isomorphisms
\begin{equation}\label{ident}
H^{\le r}\be(S_{\et}, \tau_{\leq\,  n}\e\mathbb Rf_{\lbe *}\bg_{m,X})\simeq H^{\le r}\be(S_{\et},\mathbb Rf_{\lbe *}\bg_{m,X})\simeq H^{\le r}(X_{\et},\bg_{m,\le X}\lbe).
\end{equation}
Cf. \cite[Corollary 3.2.3.3, p.~85]{lip}. In particular, $H^{\le 2}\be(S_{\et}, \tau_{\leq\e  2}\e\mathbb Rf_{\lbe *}\bg_{m,\le X})=\brp\lle X$.
Further, under the identifications \eqref{ident}, $H^{\le 3}\lbe(S_{\et},\lbe i_{\lbe f}^{\le 3}\le)$ is a map 
\begin{equation}\label{h3}
H^{\le 3}\lbe(S_{\et},\lbe i_{\lbe f}^{\le 3}\le)\colon H^{\e 3}\lbe(S_{\et},\tau_{\leq\e 2}\mathbb Rf_{\lbe *}\bg_{m,X})\hookrightarrow H^{\le 3}\lbe(X_{\et},\bg_{m,\le X}), 
\end{equation}
where the injectivity comes from \eqref{n1}.

Now set
\begin{equation}\label{rdiv}
\rdiv=(\tau_{\leq\e  1}\mathbb Rf_{\lbe *}\bg_{m,X})[1]\in\dbs.
\end{equation}
Recalling the definitions of $\picxs$ \eqref{picxs} and $\brxs$ \eqref{brxs1}, setting $n=1$ and $n=2$ in \eqref{n1} and shifting appropriately, we obtain distinguished triangles
\begin{equation}\label{t1}
f_{\lbe *}\bg_{m,X}[1]\to\rdiv\to\picxs\to f_{\lbe *}\bg_{m,X}[2]
\end{equation}
and
\begin{equation}\label{t2}
\rdiv[-1]\overset{i_{\be f}^{\le 2}}{\to}\tau_{\leq\e 2}\mathbb Rf_{\lbe *}\bg_{m,X}\to \brxs[-2]\to
\rdiv.
\end{equation}
We will write
\begin{equation}\label{fi}
\varphi_{\lbe X\!/\be S}\colon H^{\le 2}\lbe(S_{\et},\rdiv)\to H^{\le 3}\lbe(X_{\et},\bg_{m,\le X})
\end{equation}
for the composition 
\begin{equation}\label{bp0}
H^{\e 2}\be(S_{\et},\rdiv)\!\overset{\!H^{3}\lbe(i_{\be f}^{\le 2})}{\lra}\! H^{\e 3}\be(S_{\et},\tau_{\leq\e 2}\mathbb Rf_{\lbe *}\bg_{m,X})\overset{\!H^{3}\lbe(i_{\be f}^{\le 3})}{\hookrightarrow}H^{\e 3}\be(X_{\et},\bg_{m,\le X}),
\end{equation}
where the indicated injection is the map \eqref{h3}.

\begin{lemma}\label{coh-rdiv} Let $f\colon X\to S$ be a morphism of schemes. Then
\begin{itemize}
\item[(i)] $ H^{\le r}(S_{\et},\rdiv)=0$ for all $r\leq -2$,
\item[(ii)] $ H^{\e -1}(S_{\et},\rdiv)=\bg_{m,\le S}\le(X\le)$,
\item[(iii)] $ H^{\e 0}(S_{\et},\rdiv)=\pic X$,
\item[(iv)] $H^{\e 1}\lbe(S_{\et},\rdiv)=\bro\le$ \eqref{br1}, and
\item[(v)] there exists a canonical isomorphism of abelian groups
\[
\krn\!\!\left[H^{\e 2}\be(S_{\et},\rdiv)\overset{\!\varphi_{\be\lle X\be/\lbe S}}{\lra}H^{\e 3}\be(X_{\et},\bg_{m,\le X})\right]=\brt,
\]
where $\varphi_{\lbe X\be/\lbe S}$ is the map \eqref{fi} and $\brt$ is the group \eqref{br2}.
\end{itemize}
\end{lemma}
\begin{proof} Assertion (i) is immediate from \eqref{t1}.  Now, if $i=-1$ or 0, then
\[
H^{\le i}\lbe(S_{\et},\rdiv)= H^{\le i+1}\be(S_{\et}, \tau_{\leq\e  1}\mathbb Rf_{\lbe *}\bg_{m,X})=H^{\le i+1}\be(X_{\et},\bg_{m,X})
\]
by \eqref{ident}, whence (ii) and (iii) follow. Further, \eqref{t2} induces an exact sequence
\[
0\be\to\be  H^{\le 1}\be(S_{\et},\rdiv)\!\to\!\brp X\!\to\!\brxs(\lbe S\e)\!\to\! H^{\e 2}\lbe(S_{\et},\rdiv)\!\to\! H^{\e 3}\lbe(S_{\et},\tau_{\leq\e 2}\mathbb Rf_{\lbe *}\bg_{m,X}\be),
\]
where the last map is $H^{\le 3}\be(S_{\et},\lbe i_{\lbe f}^{\le 2}\le)$. Assertions (iv) and (v) now follow by comparing the above sequence and \eqref{bps0} and noting that $\krn\varphi_{\lbe X\!/\be S}=\krn H^{\le 3}\be(S_{\et},\lbe i_{\lbe f}^{\le 2}\le)$ by \eqref{bp0}.
\end{proof}

We now define the (\'etale) {\it complex of relative units of $X$ over $S$} by
\begin{equation}\label{ruxs}
\ruxs=C^{\e\bullet}\be(f^{\e\flat})\in\dbs,
\end{equation}
where $f^{\e\flat}\colon \bg_{m,\le S}\to f_{\lbe *}\bg_{m,\le X}$ is the map \eqref{fb}. Clearly, $\ruxs$ fits into a distinguished triangle
\[
\bg_{m,\le S}\overset{\!f^{\le\flat}}{\to}f_{\lbe *}\bg_{m,\le X}\to \ruxs\to \bg_{m,\le S}[\le 1\le].
\]
The {\it \'etale sheaf of relative units of $X\!$ over $S$} is
\[
\uxs=H^{\le 0}\lbe(\ruxs)=\cok\be f^{\e\flat}.
\]
We will write 
\begin{equation}\label{uxss}
\uxss=\uxs\lbe(S\le).
\end{equation}
Now let
\begin{equation}\label{fnat}
f^{\nat}\colon\bg_{m,\le S}\to \tau_{\leq\e  1}\mathbb Rf_{\lbe *}\bg_{m,X}
\end{equation}
be the following composition of morphisms in $\dbs$:
\begin{equation}\label{cump}
\bg_{m,\le S}\overset{\,f^{\le\flat}}\to f_{\lbe *}\bg_{m,\le X}\overset{\!\be i_{\lbe f}^{\le  1}}{\lra}
\tau_{\leq\e  1}\mathbb Rf_{\lbe *}\bg_{m,X},
\end{equation}
where $i_{\be f}^{\le  1}$ is the map \eqref{ifn} for $n=1$. The {\it relative units-Picard complex of $X$ over $S\,$} is the following object of $\dbs$:
\begin{equation}\label{upic}
\upicxs=C^{\e\bullet}\lbe(\le f^{\nat}\le)[1]=C^{\e\bullet}\lbe(\le f^{\nat}[1]\le).
\end{equation}

\begin{remark} \label{res-tr} If $\alpha\colon T\to S$ is a finite, \'etale and surjective morphism of schemes of constant degree $d$, then it follows from \eqref{inv} that $\alpha^{\lbe*}\e\upicxs={\rm UPic}_{X_{T}\be/\le T}$ in $D^{\le b}(T_{\et})$. On the other hand, \cite[IX, \S5.1]{sga4} implies that there exist canonical restriction and corestriction (or trace) morphisms (in $\dbs$) $\res\colon \upicxs\to \alpha_{*}\alpha^{\lbe*}\e\upicxs$ and $\tr\colon \alpha_{*}\alpha^{\lbe*}\e\upicxs\to\upicxs$  such that $\tr\circ\res$ is the multiplication by $d$ map on $\upicxs$. Consequently, for every integer $r\geq 0$, there exist canonical restriction and corestriction homomorphisms of abelian groups $\res\colon H^{\le r}\lbe(S_{\et},\upicxs\le)\to H^{\le r}\lbe(T_{\et},{\rm UPic}_{X_{T}\be/\le T}\le)$ and  $\tr\colon H^{\le r}\lbe(T_{\et},{\rm UPic}_{X_{T}\be/\le T}\le)\to H^{\le r}\lbe(S_{\et},\upicxs\le)$ such that $\tr\circ\res$ is the multiplication by $d$ map on $H^{\le r}\lbe(S_{\et},\upicxs\le)$. 
\end{remark}

By \eqref{rdiv}, \eqref{fnat} and \eqref{upic}, there exists a canonical distinguished triangle in $\dbs$:
\begin{equation}\label{t3}
\bg_{m,S}[1]\overset{\! f^{\nat}\lbe[1]}{\lra}\rdiv\to\upicxs\to \bg_{m,S}[\le 2\le].
\end{equation}
Under the identifications of Lemma \ref{coh-rdiv}, (iii) and (iv), the maps $H^{\le 0}\lbe(S_{\et},f^{\nat}\lbe[1])$ and $H^{\le 1}\lbe(S_{\et},f^{\nat}\lbe[1])$  induced by \eqref{t3} are identified with
\begin{equation}\label{h1f}
H^{\le 0}\be(S_{\et},f^{\nat}\lbe[1])=\pic\be f
\end{equation}
and
\begin{equation}\label{h2f}
H^{\le 1}\be(S_{\et},f^{\nat}\lbe[1])=\br^{\be *}\lbe f,
\end{equation}
where $\br^{\be *}\be f$ is the map \eqref{brf1}.

\begin{proposition}\label{is} Let $f\colon X\to S$ be a morphism of schemes. Then there exist canonical exact sequences of abelian groups
\begin{equation}\label{two}
0\to\npicxs\to H^{\le 0}\lbe(S_{\et},\upicxs)\to\bxs\to 0
\end{equation}
and
\begin{equation}\label{thr3}
0\to \bra\to H^{\le 1}\lbe(S_{\et},\upicxs)\to \krn H^{\le 2}\lbe(S_{\et},f^{\nat}\lbe[1])\to 0,
\end{equation}
where $\npicxs$, $\bxs$ and $\bra$ are the groups \eqref{npic}, \eqref{relb} and \eqref{bra}, respectively, and $f^{\nat}$ is the map \eqref{fnat}.
\end{proposition}
\begin{proof} Via the identifications \eqref{h1f} and \eqref{h2f}, the triangle \eqref{t3} induces an exact sequence of abelian groups  
\[
\begin{array}{rcl}
\pic S\overset{\!\pic\be f}{\lra} \pic\lbe X&\to& H^{\e 0}(S_{\et},\upicxs)\to\brp S\overset{\be\!\br^{\be *}\! f}{\lra}\bro\\\\
&\to& H^{\e 1}(S_{\et},\upicxs)\to \krn H^{\le 2}\lbe(S_{\et},f^{\nat}\lbe[1])\to 0.
\end{array}
\]
The sequences \eqref{two} and \eqref{thr3} are extracted from the above sequence using the definitions \eqref{npic}, \eqref{relb} and \eqref{bra}. 
\end{proof}

\begin{corollary}\label{nco} Let $f\colon X\to S$ be a morphism of schemes.
\begin{itemize}
\item[(i)] If $\brp\le S=0$, or if $I(\le f\le)$ is defined and is equal to $1$, then $\npicxs= H^{\le 0}\lbe(S_{\et},\upicxs)$.
\item[(ii)] If $H^{\le 3}\lbe(S_{\et},\bg_{m,S})=0$, or if $I(\le f\le)$ is defined and is equal to $1$, then $\bra=H^{\le 1}\lbe(S_{\et},\upicxs)$.
\end{itemize}
\end{corollary}
\begin{proof} Clearly, $\bxs\subseteq \brp S$. Now, since $H^{\le 2}\lbe(S_{\et},f^{\nat}\lbe[1])=H^{\le 3}\lbe(S_{\et},i_{\be f}^{\le 1}\be\circ\be f^{\le\flat}\e)$, the composite map
\[
H^{\le 3}\lbe(S_{\et},\bg_{m,\le S})\overset{\! H^{2}\lbe(\lle f^{\nat}\lbe[1])}{\lra}H^{\e 2}\lbe(S_{\et},\rdiv)\!\overset{\!\varphi_{\lbe X\!\lbe/\be S}}{\lra}\! H^{\e 3}\lbe(X_{\et},\bg_{m,\le X})
\]
is the pullback $f^{(3)}\colon H^{\le 3}\lbe(S_{\et},\bg_{m,\le S})\to H^{\le 3}\lbe(X_{\et},\bg_{m,\le X})$. Thus
$\krn H^{\le 2}\lbe(S_{\et},f^{\nat}\lbe[1])\subseteq \krn f^{(3)}\subseteq H^{\le 3}\lbe(S_{\et},\bg_{m,\le S})$. Further, if $I(\le f\le)$ is defined, then $\bxs$ and $\krn f^{(3)}$ are annihilated by $I(\le f\le)$ by \cite[Remark 3.1(d)]{ga}. The corollary is now immediate from the proposition.
\end{proof}

\begin{remark} The condition $H^{\le 3}\lbe(S_{\et},\bg_{m,S})=0$ in part (ii) of the corollary holds true if $S$ is the spectrum of a global field $k$ or if $S$ is a non-empty open affine subscheme of the spectrum of the ring of integers of a number field $k$ or of the unique smooth, complete and irreducible curve over a finite field with (global) function field $k$. See \cite[Remark II.2.2(a), p.~165]{adt}. When $S=\spec k$, where $k$ is a number field, and $X$ is a smooth and projective $k$-variety (so that ${\rm UPic}_{\lbe X\be/k}=\pic_{\!\be X\be/k}$ by Remark \ref{ru}(d) below), the isomorphism in assertion (ii) of the corollary is the isomorphism $\braxk= H^{1}(k, \pic \xs\e)$ which appears often in the theory of the Brauer--Manin obstruction (where it is usually derived from the Hochschild-Serre spectral sequence).
\end{remark}

\smallskip

Next, an application of Lemma \eqref{kc-c} to the pair of maps in \ref{cump}, using \eqref{coni}, yields a canonical distinguished triangle in $\dbs$:
\begin{equation}\label{t4}
\ruxs[1]\to\upicxs\to \picxs\to \ruxs[\le 2\le],
\end{equation}
where $\ruxs$ is the complex of relative units \eqref{ruxs}. If $f\colon X\to S$ is {\it schematically dominant}, then $f^{\e\flat}$ is injective by \cite[Lemma 2.4]{ga} and therefore $\ruxs=\uxs$. In this case, \eqref{t4} is a distinguished triangle 
\begin{equation}\label{t5}
\uxs[1]\to\upicxs\to \picxs\to \uxs[\le 2\le].
\end{equation}
Consequently
\begin{equation}\label{up-1}
H^{-1}\lbe(\e\upicxs)=\uxs
\end{equation}
and
\begin{equation}\label{up-0}
H^{\le 0}\lbe(\e\upicxs)=\picxs.
\end{equation}

The following statement is a broad generalization of \cite[Lemma 6.3]{san}.

\begin{proposition}\label{upic-sd}  Let $f\colon X\to S$ be a schematically dominant morphism of schemes.
\begin{enumerate}
\item[(i)] $H^{\le r}(S_{\et},\upicxs)=0$ for all $r\leq -2$.
\item[(ii)] $ H^{-1}(S_{\et},\upicxs)=\uxss$ \eqref{uxss}.
\item[(iii)] If $r\geq 0$, there exists a canonical exact sequence of abelian groups
\[
\dots\to  H^{\le r+1}\lbe(S_{\et},\uxs)\to  H^{\le r}\be(S_{\et},\upicxs)\to H^{\le r}\be(S_{\et},\picxs)\to H^{\le r+2}\lbe(S_{\et},\uxs)\to\dots.
\]
\item[(iv)] If $I(\le f\le)$ is defined and is equal to $1$, then
there exists a canonical exact sequence of abelian groups
\[
\begin{array}{rcl}
0&\to & H^{1}(S_{\et},\uxs)\to\npicxs\to\picxs(S\e)\to
H^{2}(S_{\et},\uxs)\\\\
&\to& \bra\to H^{1}(S_{\et},\picxs)\to H^{3}(S_{\et},\uxs).
\end{array}
\]
\item[(v)] If $H^{\le 3}\lbe(S_{\et},\bg_{m,S})=0$, then
there exists a canonical exact sequence of abelian groups
\[
\picxs(S\e)\to H^{2}(S_{\et},\uxs)\to \bra\to H^{1}(S_{\et},\picxs)\to H^{3}(S_{\et},\uxs).
\]
\end{enumerate}
\end{proposition}
\begin{proof} Assertions (i)-(iii) follow from the $S_{\et}$-cohomology sequence induced by the triangle \eqref{t5}. Assertions (iv) and (v) follow from (iii) using the 
identifications $\npicxs= H^{\le 0}\lbe(S_{\et},\upicxs)$ and $\bra=H^{\le 1}\lbe(S_{\et},\upicxs)$ of Corollary \ref{nco}. 
\end{proof}

\begin{remarks}\label{ru}\indent
\begin{enumerate}
\item[(a)] If $f$ is schematically dominant and $\uxs=0$, then $\upicxs=\picxs$ by \eqref{t5}. 
\item[(b)] Let $f\colon X\to S$ be a quasi-compact and quasi-separated morphism of schemes such that the canonical map $\s O_{\be S}\to f_{\lbe *}\s O_{\be X}$ is an isomorphism of Zariski sheaves on $S$. If $T\to S$ is any flat morphism of schemes, then $f_{T}\colon X_{T}\to T$ is quasi-compact and quasi-separated and 
$\s O_{\lbe T}\to (f_{T})_{*}\s O_{\be X_{T}}$ is an isomorphism of 
Zariski sheaves on $T$ \cite[${\rm III}_{1}$, Proposition 1.4.15, and ${\rm IV}_{1}$, (1.7.21)]{ega}. It follows that $f_{T}^{\le\flat}\colon \bg_{m,\le T}\to (f_{T}\be)_{*}\bg_{m,\e X_{T}}$ is an isomorphism of \'etale sheaves on $T$. Thus ${\rm RU}_{\be X_{\lbe T}\lbe/\le T}$ \eqref{ruxs} is the zero object of $D^{b}(T_{\et})$. In particular, $U_{\be X_{T}\lbe/\le T}=0$.  
		
\item[(c)] Let $f\colon X\to S$ be a proper and dominant morphism of locally noetherian integral schemes with geometrically integral generic fiber. Assume, in addition, that $S$ is normal (in particular $S$ is reduced, whence $f$ is schematically dominant). By \cite[${\rm IV}_{2}$, Corollary 4.6.3]{ega}, the function field of $S$ is algebraically closed in the function field of $X$. Thus, by \cite[${\rm III}_{1}$, Corollary 4.3.12]{ega}, the canonical map of Zariski sheaves $\s O_{\be S}\to f_{\be *}\le\s O_{\be X}$ is an isomorphism. Since $f$ is quasi-compact and quasi-separated, we conclude from (b) that $U_{\be X_{\lbe T}\lbe/\le T}=0$ for every flat morphism $T\to S$. Now (a) yields the equality $\upicxs=\picxs$.

\item[(d)] Let $X$ be a smooth and projective $k$-variety. Then the structural morphism $f\colon X\to \spec k$ satisfies the conditions stated in (c), whence $U_{\be X\be/k}=0$ and therefore ${\rm UPic}_{\lbe X\be/k}=\pic_{\!\be X\be/k}$. See Remark \ref{comm} below.
\end{enumerate}
\end{remarks}

\begin{lemma}\label{split1} Let $f\colon X\to S$ be a morphism of schemes. If $f\colon\be X\to S$ has a section, then the triangle \eqref{t3} splits, i.e., is isomorphic to the triangle
\[
\bg_{m,S}[1]\to \bg_{m,S}[1]\oplus\upicxs\to\upicxs\overset{\be 0}\to
\bg_{m,S}[\le 2\le]
\]
with canonical unlabeled maps. Consequently,
\[
H^{\le r}\be(S_{\et},\rdiv)\simeq H^{\le r+1}\be(S_{\et},\bg_{m,\le S})\be\oplus \! H^{\lle r}\lbe(S_{\et},\upicxs)
\]
for every $r\geq -1$.
\end{lemma}
\begin{proof} Let $\sigma\colon S\to X$ be a section of $f$. 
The composition of canonical morphisms in $D(S_{\et})$
\[
\mathbb R\le f_{\lbe *}\bg_{m,\le X}\overset{\mathbb R\le f_{ *}\be(\sigma^{\flat}\lbe)}{\lra}\mathbb R\le f_{\lbe *}\lbe(\le \sigma_{\lbe *}\bg_{m,\le S})\to \mathbb R\le f_{\lbe *}(\le \mathbb R\e \sigma_{\lbe *}\bg_{m,\le S})\simeq \mathbb R(\le f_{\le *}\!\circ\be \sigma_{\lbe *})\bg_{m,S}=\bg_{m,S}
\]
induces a morphism $\alpha\colon\tau_{\leq\e 1}\mathbb R\le f_{\lbe *}\bg_{m,\le X}\to\bg_{m,\le S}$ in $\dbs$ such that the composition
\[
f_{*}\bg_{m,X}\overset{\!i_{\be f}^{1}}{\lra}\tau_{\leq\e 1}\mathbb R\le f_{\lbe *}\bg_{m,\le X}\overset{\!\alpha}{\lra}\bg_{m,\le S}
\]
equals $f_{\lbe *}(\sigma^{\le\flat})$. Thus, since $f^{\nat}$ is the composition of the maps in \eqref{cump}, the equation \eqref{marv2} shows that the map $\alpha[\le 1\le]\colon\rdiv\to\bg_{m,\le S}[1]$
is a retraction of the map $f^{\nat}\lbe[1]\colon \bg_{m,\le S}[1]\to\rdiv$ appearing in \eqref{t3}. The lemma now follows from \cite[Remark 1.2.9, p.~44]{nee}.
\end{proof}

\begin{proposition}\label{tor} Let $f\colon X\to S$ be a morphism of schemes such that the \'etale index $I(\le f\le)$ is defined and is equal to $1$. Then there exists a canonical exact sequence of abelian groups
\[
0\to\brt\to H^{\le 2}\lbe(S_{\et},\upicxs\le)\to H^{\le 3}\lbe(X_{\et},\bg_{m,\le S}\lbe)/H^{\le 3}\lbe(S_{\et},\bg_{m,\le S}\lbe),
\]
where $\brt$ is the group \eqref{br2}.
\end{proposition}
\begin{proof} Since $f^{(3)}\colon H^{\le 3}\lbe(S_{\et},\bg_{m,\le S})\to H^{\le 3}\lbe(X_{\et},\bg_{m,\le X})$ factors as
\[
H^{\le 3}\lbe(S_{\et},\bg_{m,\le S})\overset{\! H^{2}\lbe(\lle f^{\nat}\lbe[1])}{\lra}H^{\e 2}\lbe(S_{\et},\rdiv)\!\overset{\!\varphi_{\lbe X\!\lbe/\be S}}{\lra}\! H^{\e 3}\lbe(X_{\et},\bg_{m,\le X}),
\]
where $\varphi_{\be X\!/\be S}$ is the map \eqref{fi} (see the proof of Corollary \ref{nco}), an application of Lemma~\ref{ker-cok} to the above pair of maps using Lemma \ref{coh-rdiv}(v) yields
an exact sequence 
\[
0\to \krn H^{\le 2}\lbe(S_{\et},f^{\nat}\lbe[1])\to \krn f^{(3)}\to \brt\to \cok \be H^{\le 2}\be(S_{\et},f^{\nat}\lbe[1])\to\cok f^{(3)}.
\]
The hypothesis $I(\le f\le)=1$ yields $\krn f^{(3)}=0$ (see \cite[Remark 3.1(d)]{ga}). Further, we may write 
$\cok f^{(3)}=H^{\le 3}\lbe(X_{\et},\bg_{m,\le S}\lbe)/H^{\le 3}\lbe(S_{\et},\bg_{m,\le S}\lbe)$. 
We will complete the proof by showing that there exists a canonical isomorphism of abelian groups $\cok \be H^{\le 2}\lbe(S_{\et},f^{\nat}\lbe[1])=H^{\le 2}\lbe(S_{\et},\upicxs\le)$. The triangle \eqref{t3} induces an isomorphism
\[
\cok \be H^{\le 2}\be(S_{\et},f^{\nat}\lbe[1])\simeq\krn\!\!\left[H^{\le 2}\lbe(S_{\et},\upicxs\le)\overset{\!\partial_{\lbe X\!/\lbe S}}{\lra} H^{\le 4}\lbe(S_{\et},\bg_{m,\le S})\right]
\]
where, by Lemma \ref{split1}, the connecting map $\partial_{\lbe X\!/\lbe S}$ is zero if $f$ has a section. 
Let $T\to S$ be a finite \'etale quasi-section of $f$ of constant degree $d$ and consider the commutative diagram with exact rows
\[
\xymatrix{H^{\le 2}\lbe(S_{\et},\upicxs\le)\ar[d]^{\res}\ar[r]& \img\partial_{X/S}\,\ar[d]^{\res}\ar[r]& 0\\
H^{\le 2}\lbe(T_{\et},{\rm UPic}_{X_{T}\be/\le T}\lbe)\ar[d]^{\tr}\ar[r]^(.6){0}& \img\partial_{X_{\lbe T}\be/\le T}\,\ar[d]^{\tr}\ar[r]& 0\\
H^{\le 2}\lbe(S_{\et},\upicxs\le)\ar[r]& \img\partial_{X/S}\ar[r]& 0,
}
\]
where the maps on the right-hand column are induced by those on the left-hand column. The left-hand (and therefore also the right-hand) vertical composition  is the multiplication by $d$ map (see Remark \eqref{res-tr}). It follows that $\img\partial_{X/S}$ is annihilated by $I\lbe(\lle f\lle)$, which completes the proof.
\end{proof}

\begin{remark}\label{comm} The proposition generalizes \cite[Proposition 1.3, p.~147]{ctsk}. Note that, in constrast to [loc.cit.], the proof of the above proposition makes no appeal to the Hochschild-Serre spectral sequence and yields more information than the indicated reference in the particular case considered there (for example, the hypothesis $U_{\lbe X\be/k}=0$ in [loc.cit.] can be avoided if $\pic_{\!\be X\be/k}$ is replaced with ${\rm UPic}_{\lbe X\be/k}$. See Remark \ref{ru}(d) above). 	
\end{remark}

\section{Proof of the main theorem}

In this Section we establish the additivity theorem for $\upicxs$ mentioned in the Introduction (see Proposition \ref{upadd}) and apply it to derive our main Theorem \ref{main}.

\smallskip

Let $f\colon X\to S$ and $g\colon Y\to S$ be morphisms of schemes and recall the morphisms $p_{X}\colon X\times_{S}Y\to X$ and $p_{\le Y}\colon X\be\times_{S}\be Y\to Y$. Since $f\be\circ\be p_{\lbe X}=g\lbe\circ\lbe p_{\le Y}=f\be\times_{\be S}\be g$, the canonical maps $p_{\be X}^{\le\flat}\colon \bg_{m,\le X}\to (\e p_{\lbe X}\be)_{*}\bg_{m,\le X\lbe\times_{\be S} Y}$ and $p_{\le Y}^{\le\flat}\colon \bg_{m,\le Y}\to (\e p_{\le Y}\be)_{*}\bg_{m,\le X\lbe\times_{\be S} Y}$ induce maps $f_{\be *}\lbe(\e p_{\lbe X}^{\le\flat})\colon f_{*}\bg_{m,\le X}\be\to\be (\e f\be\times_{\be S}\lbe g\lbe)_{*}\bg_{m,\le X\lbe\times_{\be S}\lbe Y}$ and
$g_{\lbe *}\be(\e p_{\le Y}^{\le\flat})\colon g_{*}\bg_{m,\le Y}\be\to\be (\e f\be\times_{\be S}\lbe g\be)_{*}\bg_{m,\le X\lbe\times_{\be S}\lbe Y}$. Now the map
\[
\vartheta\colon f_{*}\bg_{m,\le X}\!\oplus\be g_{*}\bg_{m,\le Y}\to (\e f\be\times_{\be S}\be g\lbe)_{*}\bg_{m,\le X\lbe\times_{\be S}\lbe Y}
\]
which, on sections, is given by $(a,b)\mapsto f_{\be *}\lbe(\e p_{\lbe X}^{\le\flat})(a)\be\cdot\be g_{*}\be(\e p_{\le Y}^{\le\flat})(b)$, induces a morphism in $D(\lbe S_{\et}\be)$
\begin{equation}\label{addr2}
\mathbb R\le\vartheta\colon \mathbb R\lbe f_{*}\bg_{m,\le X}\be\oplus\be \mathbb R g_{*}\bg_{m,\le Y}\to \mathbb R (\le f\!\times_{\be S}\be g\le)_{*}\bg_{m,\le X\lbe\times_{\be S}\lbe Y}.
\end{equation}
By \eqref{marv}, the following diagram commutes
\[
\xymatrix{\bg_{m,\le S}\be\oplus\be \bg_{m,\le S}\ar[d]^(.45){(\cdot)} \ar[r]^(.44){(\le f^{\flat}\!,\e g^{\flat})}& f_{\lbe *}\bg_{m,\le X}\!\oplus\! g_{*}\bg_{m,\le Y}\ar[d]^(.44){\vartheta}\ar[r]^(.39){(\le i_{\be f}^{1},\e i_{\be g}^{1})}& \tau_{\leq\e  1}\mathbb R\lbe f_{*}\bg_{m,\le X}\be\oplus\be \tau_{\leq\e  1}\mathbb R g_{*}\bg_{m,\le Y}\ar[d]^(.44){\tau_{\leq\e  1}\lbe\mathbb R\le\vartheta}\\
\bg_{m,\le S}\ar[r]^(.35){(\le f\!\times_{\be S} g\e)^{\flat}}&(\e f\be\times_{\be S}\be g\lbe)_{*}\bg_{m,\le X\!\times_{\be S}\lbe Y}\ar[r]^(.43){i_{\be f\times g}^{1}}& \tau_{\leq\e  1}\mathbb R (\le f\!\times_{\be S}\be g\le)_{*}\bg_{m,\le X\be\times_{\be S}\lbe Y},
}
\]
where the left-hand vertical map is the multiplication morphism. Consequently, the following diagram also commutes
\[
\xymatrix{\bg_{m,\le S}\be\oplus\be \bg_{m,\le S}\ar[d]^(.45){(\cdot)} \ar[rr]^(.35){(\le f^{\nat}\!,\e g^{\nat})}&&\tau_{\leq\e  1}\mathbb R\lbe f_{*}\bg_{m,\le X}\be\oplus\be \tau_{\leq\e  1}\mathbb R g_{*}\bg_{m,\le Y}\ar[d]^(.43){\tau_{\leq\e  1}\lbe\mathbb R\le\vartheta}\\
\bg_{m,\le S}\ar[rr]^(.35){(\le f\!\times_{\be S} g\e)^{\nat}}&& \tau_{\leq\e  1}\mathbb R (\le f\!\times_{\be S}\be g\le)_{*}\bg_{m,\le X\be\times_{\be S}\lbe Y},
}
\]
where the morphisms $f^{\nat}, g^{\nat}$ and $(\le f\!\times_{\be S} g\e)^{\nat}$ are given by \eqref{fnat}. Thus, by definition \eqref{upic}, the vertical maps in the preceding diagram define a canonical morphism in $\dbs$:
\begin{equation}\label{upicm}
\upicxs\be\oplus\be\upicys\to\upicxys.
\end{equation}

We now observe that \eqref{addr2} induces a morphism in $S_{\et}^{\e\sim}$ (namely, $H^{\le r}\be(\mathbb R\le\vartheta)$):
\begin{equation}\label{rrs}
R^{\e r}\! f_{*}\bg_{m,\le X}\be\oplus\be R^{\e r}\be g_{*}\bg_{m,\le Y}\to R^{\e r}\be (\le f\!\times_{\be S}\be g\le)_{*}\bg_{m,\le X\lbe\times_{\be S}\lbe Y}
\end{equation}
for every integer $r\geq 0$. In particular, there exists an induced canonical homomorphism of abelian groups
\begin{equation}\label{ran}
\brxs(\lbe S\e)\be\oplus\be \brys(\lbe S\e)\to\brxys(\lbe S\e)
\end{equation}
which induces, in turn, a homomorphism of abelian groups
\begin{equation}\label{ran2}
\brt\be\oplus\be \brty\to\brtxy,
\end{equation}
where $\brt$ is the group defined in \eqref{br2}.

Next, we recall from \cite{ga} the homomorphisms of abelian groups
\begin{equation}\label{del2}
\nabla^{2}\colon \brp S\to \brp S\oplus \brp S,\e \xi\mapsto(\xi,\xi^{-1}\le),
\end{equation}
\begin{equation}\label{beta2}
\beta^{\le 2}=\left(f^{(\lbe 2\lbe)}\!,\e g^{(\lbe 2\lbe)}\right)\be\circ\be \nabla^{2}\colon 
\brp S\to \brp X\lbe\oplus\lbe\brp\, Y, \e \xi\mapsto (f^{(2)}(\xi),g^{(2)}(\xi)^{-1}\le),
\end{equation}
and
\begin{equation}\label{pxy2}
p_{\lbe X\lbe Y}^{\le 2}\colon \brp X\lbe\oplus\lbe\brp\, Y\to \brp(X\!\be\times_{\be S}\!\lbe Y\le),\e (a,b)\mapsto p_{\be X}^{(2)}\be(a)\cdot p_{\le Y}^{(2)}\be(b\e).
\end{equation}

\begin{remark}
If the canonical map $\br S\to\brp\le S$ is an isomorphism (see Remark \ref{bgps}), the composite map
\[
\br S\to\brp S\overset{\beta^{2}}{\to}\brp X\lbe\oplus\lbe\brp\, Y
\]
will also be denoted by $\beta^{\le 2}$. Similarly, if the canonical map  $\br Z\to\brp Z$ is an isomorphism for $Z=X,Y$ and $X\be\times_{\be S}\be Y$, the map $\br X\lbe\oplus\lbe\br\, Y\to \br\lbe(X\!\be\times_{\be S}\!\lbe Y\le)$ induced by \eqref{pxy2} will also be denoted by 
$p_{\lbe X\lbe Y}^{\le 2}$.
\end{remark}

The map \eqref{pxy2} induces homomorphisms of abelian groups
\begin{equation}\label{br1a}
\nbrxs\be\oplus\be \nbrys\to\nbrxys,
\end{equation}
\begin{equation}\label{nad}
p_{\lbe X\lbe Y\be,\, 1}^{\le 2}\colon\bro\be\oplus\be \broys\to\broxys
\end{equation}
and
\begin{equation}\label{brad}
\bra\be\oplus\be\bray\to\braxy,
\end{equation}
where the groups $\nbrxs$, $\bro$ and $\bra$ are given by \eqref{nbr}, \eqref{br1} and \eqref{bra}, respectively.

\begin{lemma}\label{b2} Let $f\colon X\to S$ and $g\colon Y\to S$ be morphisms of schemes such that the \'etale index $I\lbe(\lle f\!\times_{\be S}\be g\le)$ is defined and is equal to $1$. Then there exists a canonical exact sequence of abelian groups
\[
0\to \brp S\to\cok\beta^{\le 2}\to\nbrxs\be\oplus\be\nbrys\to 0.
\]
\end{lemma}
\begin{proof} By definition, $\beta^{\le 2}$ factors as
\[
\brp\le S\overset{\nabla^{2}}{\hookrightarrow}\brp S\!\oplus\!\brp S\overset{(\le f^{(2)}\!,\e g^{(2)})}{\hookrightarrow}\brp X\!\oplus\!\brp\, Y,
\]
where the second map is injective since $I\lbe(\lle f\lle)=I\lbe(\le g\lle)=I\lbe(\lle f\!\times_{\be S}\be g\le)=1$. The lemma now follows by applying Lemma \ref{ker-cok} to the above pair of maps noting that $\cok\nabla^{2}$ is canonically isomorphic to $\brp\le S$ via the multiplication map.
\end{proof}
 
\begin{proposition}\label{wdfl} Let $f\colon X\to S$ and $g\colon X\to S$ be morphisms of schemes. If $f\!\times_{\be S}\be g\colon X\!\times_{\be S}\be Y\to S$ has an \'etale quasi-section, then the canonical morphism \eqref{rrs}
\[
R^{\e r}\! f_{*}\bg_{m,\le X}\be\oplus\be R^{\e r}\be g_{*}\bg_{m,\le Y}\to R^{\e r}\be (\le f\!\times_{\be S}\be g\le)_{*}\bg_{m,\le X\be\times_{\be S}\lbe Y}
\]
is injective for every $r\geq 1$. In particular, the map \eqref{ran}
\[
\brxs(\lbe S\e)\be\oplus\be \brys(\lbe S\e)\to\brxys(\lbe S\e)
\]
is injective.
\end{proposition}
\begin{proof} (After \cite[Proposition 1.5]{sz}) Let $\alpha\colon T\!\to\! S$ be an \'etale quasi-section of $f\times_{\be S} g$. Since $\alpha^{*}\colon S_{\et}^{\e\sim}\to T_{\et}^{\e\sim}$ is an exact functor and $\alpha^{*}\be f_{\lbe *}\bg_{m,\le X}=(f_{T})_{*}\bg_{m,\e X_{T}}$ in $T_{\et}^{\e\sim}$ by \cite[proof of Lemma 2.5]{ga}, the edge morphism $R^{\e r}\be (f_{T})_{*}\bg_{m,\le X_{T}}\to \alpha^{*}\be R^{\e r}\be f_{*}\bg_{m,\le X}$ induced by the spectral sequence $R^{\e s}\be\alpha^{*}\be R^{\e r}\be f_{*}\bg_{m,\le X}\!\!\implies\!\! R^{\e s+r}\be (f_{T})_{*}\bg_{m,\le X_{T}}$ is an isomorphism. Analogous statements hold true when $f$ is replaced by $g$ and $f\be\times_{\be S}\lbe g$.
Thus, by \cite[IV, Corollary 4.5.8]{sga3}, it suffices to prove the proposition when $S$ is replaced with $T$. Since $f_{T}\!\times_{\be T}\! g_{\e T}$ has a section, we may further assume that $f\!\times_{\be S}\lbe g\colon X\!\times_{\be S}\be Y\to S$ has a section $\rho\le\colon\be S\to X\be\times_{\lbe S}\be Y$. Set $\sigma=p_{X}\lbe\circ\lbe\rho\colon S\to X$ and $\tau=p_{\e Y}\be\circ\be\rho\colon S\to Y$, which are sections of $f$ and $g$, respectively. By standard considerations, we only need to check that the map
\[
H^{\le r}\be(X_{\et},\be\bg_{m,\le X})\be\oplus\be H^{\le r}\be(\e Y_{\et},\be\bg_{m,\le Y})\!\to\! H^{\le r}\be((X\!\times_{\be S}\be Y)_{\et},\bg_{m,\e X\lbe\times_{\be S}\lbe Y}), (a,b)\mapsto p_{X}^{(r)}(a)\cdot p_{\le Y}^{(r)}(b),
\]
is injective if $S$ is a strictly local scheme. Let $q_{\le\tau}=1_{X}\be\times_{S}\lbe\tau\colon X\to X\be\times_{\lbe S}\lbe Y$ and $q_{\le\sigma}=\sigma\be\times_{S}\lbe 1_{Y}\colon Y\to X\be\times_{\lbe S}\lbe Y$. Since $p_{X}\circ\e q_{\le\tau}=1_{X}$ and $p_{\le Y}\circ\e q_{\le\sigma}=1_{Y}$, the compositions $q_{\le\tau}^{(r)}\circ p_{X}^{(r)}$ and $q_{\le\sigma}^{(r)}\circ p_{\le Y}^{(r)}$ are the identity maps on $H^{\le r}\be(X_{\et},\bg_{m,\le X})$ and $H^{\le r}\be(Y_{\et},\bg_{m,\le Y})$, respectively. In particular, both $p_{X}^{(r)}$ and $p_{\le Y}^{(r)}$ are injective. On the other hand, since $p_{X}\circ q_{\le\sigma}=\sigma$, $p_{\le Y}\circ q_{\le\tau}=\tau$ and the higher \'etale cohomology of a strictly local scheme is trivial, the compositions $q_{\le\sigma}^{(r)}\circ p_{X}^{(r)}$ and $q_{\le\tau}^{(r)}\circ p_{Y}^{(r)}$ are the trivial maps on $H^{\le r}\be(X_{\et},\bg_{m,\le X})$ and $H^{\le r}\be(Y_{\et},\bg_{m,\le Y})$, respectively. It follows that $p_{X}^{(r)}H^{\le r}\be(X_{\et},\bg_{m,\le X})\cap p_{\le Y}^{(r)}H^{\le r}\be(Y_{\et},\bg_{m,\le Y})=1$, where the intersection takes place inside $H^{\le r}((X\!\times_{\be S}\be Y)_{\et},\bg_{m,\e X\lbe\times_{\be S}\lbe Y})$. We conclude that $p_{X}^{(r)}(a)\cdot p_{\le Y}^{(r)}(b)=1$ if, and only if, $p_{X}^{(r)}(a)=p_{\le Y}^{(r)}(b)=1$. Since both maps $p_{X}^{(r)}$ and $p_{\le Y}^{(r)}$ are injective, the proposition follows.
\end{proof}

\begin{proposition}\label{upadd} Let $S$ be a locally noetherian normal scheme and let $f\colon X\to S$ and $g\colon Y\to S$ be faithfully flat morphisms locally of finite type. Assume that
\begin{enumerate}
\item[(i)] $f\be\times_{\be S}\be g$ has an \'etale quasi-section,
\item[(ii)] the strict henselisations of the local rings of $X$, $Y$ and $X\be\times_{S}\be Y$ are factorial, 
\item[(iii)] for every point $s\in S$ of codimension $\leq 1$, the fibers $X_{\lbe s}$ and $Y_{\lbe s}$ are geometrically integral and
\item[(iv)] for every maximal point $\eta$ of $S$,
\begin{enumerate}
\item[(a)] ${\rm gcd}\le(\le I(X_{\eta}\le),I(\le Y_{\eta}\le))=1$ and
\item[(b)] $\pic\be(X_{\lbe\eta}^{\le\rm s}\be\times_{k(\eta)^{\rm s}}\be Y_{\!\eta}^{\e\rm s}\le)^{\g(\eta)}=(\e\pic X_{\lbe\eta}^{\le\rm s})^{\g(\eta)}\be\oplus\be(\e\pic Y_{\!\eta}^{\e\rm s})^{\g(\eta)}$.
\end{enumerate}
\end{enumerate}
Then the canonical map \eqref{upicm}
\[
\upicxs\oplus\upicys\to\upicxys
\]
is an isomorphism in $\dbs$.
\end{proposition}
\begin{proof} Note that $f,g$ and $f\!\times_{\be S}\! g$ are schematically dominant and therefore \eqref{up-1} and \eqref{up-0} are valid for each of these morphisms. Now, by (i), (iii) and \cite[Corollary 4.3]{ga}, the map $\uxs\be\oplus\be\uys\to\uxys$, i.e., $H^{-1}\lbe(\e\upicxs\be\oplus\be\upicys)\to H^{-1}\lbe(\le\upicxys)$ \eqref{up-1}\le, is an isomorphism of \'etale sheaves on $S$. On the other hand, since $\picxs$ is the \'etale sheaf on $S$ associated to the presheaf $T\mapsto\npic\!(X_{T}\lbe/\e T\le)$ \cite[Definition 9.2.2, p.~252]{klei}, in order to check that $H^{\le 0}\lbe(\e\upicxs\le\oplus\le\upicys)\to H^{\le 0}\lbe(\le\upicxys)$, i.e., $\picxs\oplus\picys\to\picxys$ \eqref{up-0}, is an isomorphism of \'etale sheaves on $S$, it suffices to check that, for every \'etale and surjective morphism $h\colon T\to S$, the canonical map
\begin{equation}\label{ut}
\npic\be(X_{\lbe T}\lbe/\e T\le)\oplus \npic\be(\e Y_{\lbe T}\lbe/\e T\le)\mapsto \npic\be(X_{\lbe T}\!\times_{T}\! Y_{\lbe T}\lbe/\e T\e)
\end{equation}
is an isomorphism of abelian groups. By (ii), $X_{\lbe T}, Y_{\lbe T}$ and $X_{\lbe T}\!\times_{T}\! Y_{\lbe T}$ are locally factorial. Further, it is not difficult to verify that hypotheses (iii) and (iv) are stable under the base change $h\colon T\to S$ (this verification makes use of the following observation: if $\eta^{\e\prime}$ is a maximal point of $T$, then $\eta=h(\eta^{\e\prime}\le)$ is a maximal point of $S$ and $I(\le (\lbe X_{\lbe T}\lbe)_{\eta^{\le\prime}}\lbe)\!\mid\! I(X_{\eta}\le)$ by \eqref{genf-0}). Now Proposition \ref{opo2} shows that \eqref{ut} is an isomorphism.
\end{proof}

The following corollary of the proposition is a variant of \cite[Lemma 6.6(ii)]{san}.

\begin{corollary}\label{bradd} Let $S$ be a locally noetherian regular scheme and let $f\colon X\to S$ and $g\colon Y\to S$ be smooth and surjective morphisms. Assume that	
\begin{enumerate}
\item[(i)] either
\begin{enumerate}
\item[(a)] $H^{\le 3}\lbe(S_{\et},\bg_{m,S})=0$ or
\item[(b)] $I(\le f\!\times_{\be S}\be g\le)$ is defined and is equal to $1$,
\end{enumerate}
\item[(ii)] for every point $s\in S$ of codimension $\leq 1$, the fibers $X_{\lbe s}$ and $Y_{\lbe s}$ are geometrically integral, and
\item[(iii)] for every maximal point $\eta$ of $S$,
\begin{enumerate}
\item[(a)] ${\rm gcd}\le(\le I(X_{\eta}\le),I(\le Y_{\eta}\le))=1$ and
\item[(b)] $\pic\be(X_{\lbe\eta}^{\le\rm s}\be\times_{k(\eta)^{\rm s}}\be Y_{\!\eta}^{\e\rm s}\le)^{\g(\eta)}=(\e\pic X_{\lbe\eta}^{\le\rm s})^{\g(\eta)}\be\oplus\be(\e\pic Y_{\!\eta}^{\e\rm s})^{\g(\eta)}$.
\end{enumerate}
\end{enumerate}	
Then the canonical map \eqref{brad}
\[
\bra\be\oplus\be\bray\to\braxy
\]
is an isomorphism of abelian groups.
\end{corollary}
\begin{proof} Hypothesis (i) of Proposition \ref{upadd} clearly holds since $f\!\times_{\be S}\lbe g$ is smooth. On the other hand, since $S$ is regular and $f$ and $g$ are smooth, $X,Y$ and $X\be\times_{\lbe S}\be Y$ are regular as well \cite[${\rm IV}_{4}$, Proposition 17.5.8(iii)]{ega}, whence the strict henselization of every local ring of $X,Y$ and $X\be\times_{\lbe S}\be Y$ is regular (and therefore factorial) \cite[Lemma 15.41.10]{sp}.
Thus hypothesis (ii) of Proposition \ref{upadd} also holds. Thus all the hypotheses of Proposition \ref{upadd} hold, whence the canonical morphism 
$\upicxs\oplus\upicys\to \upicxys$ is an isomorphism in $\dbs$ which induces an isomorphism of abelian groups $H^{\le 1}\be(S_{\et},\upicxs)\oplus H^{\le 1}\be(S_{\et},\upicys)\isoto H^{\le 1}\be(S_{\et},\upicxys)$.
Finally, since $I(\le f\le)$ and $I(\e g\le)$ both divide $I(\le f\!\times_{\be S}\lbe g\le)$, hypothesis (i) and Corollary \ref{nco}(ii) together show that $\br_{\!\rm a}^{\prime}(Z/S)= H^{\le 1}\be(S_{\et},{\rm UPic}_{Z/S})$ for $Z=X, Y$ and $X\times_{S}Y$, whence the corollary follows.
\end{proof}

\begin{proposition}\label{bro!} Let $S$ be a locally noetherian regular scheme and let $f\colon X\to S$ and $g\colon Y\to S$ be smooth and surjective morphisms. Assume that	
\begin{enumerate}
\item[(i)] the \'etale index $I\lbe(\le f\!\times_{\be S}\be g\le)$ is defined and is equal to $1$, 
\item[(ii)] for every point $s\in S$ of codimension $\leq 1$, the fibers $X_{\lbe s}$ and $Y_{\lbe s}$ are geometrically integral, and
\item[(iii)] for every maximal point $\eta$ of $S$,
\[
\pic\be(X_{\lbe\eta}^{\le\rm s}\be\times_{k(\eta)^{\rm s}}\be Y_{\!\eta}^{\e\rm s}\le)^{\g(\eta)}=(\e\pic X_{\lbe\eta}^{\le\rm s})^{\g(\eta)}\be\oplus\be(\e\pic Y_{\!\eta}^{\e\rm s})^{\g(\eta)}.
\]
\end{enumerate}	
Then there exist canonical exact sequences of abelian groups
\[
0\to\brp S\to\bro\be\oplus\be\broys\to\broxys\to 0,
\]
where the first nontrivial map is induced by \eqref{beta2} and the second nontrivial map is \eqref{nad}, and
\begin{equation}\label{brob}
\begin{array}{rcl}
0\to\nbrxs\be\oplus\be \nbrys\to\nbrxys&\to&\,\displaystyle\frac{\brxys\lbe(\lbe S\e)}{\brxs(\lbe S\e)\!\oplus\! \brys(\lbe S\e)}\\\\
&\to& \displaystyle\frac{\brtxy}{\brt\!\oplus\!\brty}\to 0,
\end{array}
\end{equation}
where the first nontrivial map is \eqref{br1a}.
\end{proposition}
\begin{proof} The sequence \eqref{br-seq} induces commutative diagrams of abelian groups with exact rows
\[
\xymatrix{0\ar[r]&\bxs\be\oplus\be \bys\ar[d]^(.45){\simeq}\ar[r]& \brp\le S\!\oplus\!\brp\le  S\ar@{->>}[d]^(.45){(\e\cdot\e)}\ar[r]& \img\br^{\be *}\be f\!\oplus\!\img\br^{\be *}\be g\ar@{->>}[d]\ar[r]& 0\\
0\ar[r]&\brp(X\!\times_{\be S}\! Y\be/\be S\le)\ar[r]& \brp\le  S\ar[r]& \img \br^{*}\lbe(f\be\times_{\be S}\be g\le)\ar[r]& 0.
}
\]
and
\[
\xymatrix{0\ar[r]&\img\br^{\be *}\be f\!\oplus\!\img\br^{\be *}\be g\,\ar@{->>}[d]\ar[r]& \bro\!\oplus\!\broys\ar@{->>}[r]\ar[d]& \bra\!\oplus\!\bray\ar[d]^{\simeq}\ar[r]& 0\\
0\ar[r]&\img\br^{*}\be(f\be\times_{\be S}\be g\le)\ar[r]& \broxys\ar@{->>}[r]& \braxy\ar[r]& 0.
}
\]
In the first diagram, the left-hand vertical map is an isomorphism by (i) and \cite[Proposition 3.2]{ga}. The middle vertical map in the first diagram is clearly surjective and its kernel can be identified with $\brp\le S$ via the map \eqref{del2}. In the second diagram, the right-hand vertical map is an isomorphism by Corollary \ref{bradd} (note that, by \eqref{genf}, (i) implies hypothesis (iii)(a) of Corollary \ref{bradd}). The first sequence of the proposition now follows by applying the snake lemma to both diagrams above.

Next, by (i), $I\lbe(\lle f\lle)=I\lbe(\e g\lle)=1$. Thus, by Proposition \ref{tor}, there exists a canonical commutative diagram with exact rows
\begin{equation}\label{bab}
\xymatrix{0\ar[r]&\brt\be\oplus\be \brty\ar[d]\ar[r]& H^{\le 2}\lbe(S_{\et},\upicxs)\be\oplus\! H^{\le 2}\lbe(S_{\et},\upicys)\ar[d]^{\simeq}\\
0\ar[r]&\brtxy\ar[r]&H^{\le 2}\lbe(S_{\et},\upicxys),
}
\end{equation}
where the right-hand vertical map is an isomorphism by Proposition \ref{upadd} (see the proof of Corollary \ref{bradd}). Next we will split the four-column commutative diagram with exact rows induced by the sequences \eqref{bps1} for $f, g$ and $f\!\times_{\be S}\!g$. Set
\[
\hspace{0.7cm}I=\krn\!\be\left[\e\brxs(\lbe S\e)\!\oplus\! \brys(\lbe S\e)\twoheadrightarrow \brt\be\oplus\be \brty\e\right]
\]
and
\[
J=\krn\!\be\left[\e\brxys(\lbe S\e)\twoheadrightarrow \brtxy\e\right],
\]
where the indicated maps are induced by the projections in the sequences \eqref{bps1} for $f,g$ and $f\be\times_{\be S}\be g$. Then the rows of the following commutative diagrams are exact
\begin{equation}\label{a0}
\xymatrix{0\ar[r]& \bra\!\oplus\!\bray\ar[r]\ar[d]^{\simeq}&\nbrxs\!\oplus\!\nbrys\ar[r]\ar[d]& I\ar[r]\ar@{^{(}->}[d]& 0\\
0\ar[r]& \braxy\ar[r]&\nbrxys\ar[r]& J\ar[r]& 0
}
\end{equation}
and
\begin{equation}\label{b0}
\xymatrix{0\ar[r]& I\ar[r]\ar@{^{(}->}[d]&\brxs(\lbe S\e)\!\oplus\! \brys(\lbe S\e)\ar[r]\ar@{^{(}->}[d]& \brt\!\oplus\!\brty\ar[r]\ar@{^{(}->}[d]& 0\\
0\ar[r]& J\ar[r]&\brxys(\lbe S\e)\ar[r]& \brtxy\ar[r]& 0.
}
\end{equation}
The middle vertical map in \eqref{b0} is injective since, by smoothness, $f\!\times_{\be S}\!g$ has an \'etale quasi-section and Proposition \ref{wdfl} applies. On the other hand, diagram \eqref{bab} shows that the third vertical map in \eqref{b0} is injective. We conclude that the first vertical map in \eqref{b0} (which is the same as the third vertical map in \eqref{a0}) is injective. The left-hand vertical map in \eqref{a0} is an isomorphism by Corollary \ref{bradd}. The sequence \eqref{brob} now follows by applying the snake lemma to diagrams \eqref{a0} and \eqref{b0} and assembling the resulting exact sequences.
\end{proof}

The following statement is an immediate consequence of the exactness of \eqref{brob}.

\begin{corollary} \label{cool} Let the notation and hypotheses be those of the proposition. If $\brxys\lbe(\lbe S\e)=\brxs(\lbe S\e)\be\oplus\be \brys(\lbe S\e)$, then the canonical map 
\[
\nbrxs\be\oplus\be \nbrys\to\nbrxys
\]
is an isomorphism. Consequently, there exists a canonical exact sequence of abelian groups
\[
0\to\brp S\to \brp(X\!\be\times_{\be S}\! Y\le)\to\nbrxs\be\oplus\be \nbrys\to 0.
\]
\end{corollary}

\begin{remark} If $S=\spec k$, where $k$ is a field, and the morphisms $f\colon X\to\spec k$ and $g\colon X\to\spec k$ are quasi-compact and quasi-separated, the groups $I$ and $J$ which appear in the proof of Proposition \ref{bro!} are canonically isomorphic, respectively, to the direct sum of the transcendental Brauer groups of $X$ and $Y$ and the transcendental Brauer group of $X\be\times_{ k}\lbe Y$ defined in Remark \ref{int}(c). Thus the proof of the proposition yields the following: if $X$ and $Y$ are smooth $k$-varieties such that $I(\le X\be\times_{\le k}\lbe Y\le)=1$ and $\pic\be(X^{\le\rm s}\be\times_{\ks}\be Y^{\e\rm s}\le)^{\g}=(\e\pic X^{\le\rm s})^{\g}\be\oplus\be(\e\pic Y^{\e\rm s})^{\g}$, then the canonical map of transcendental Brauer groups
\[
\br_{\be{\rm t}}^{\le\prime}\lbe(\be X\!/k)\oplus \br_{\be{\rm t}}^{\le\prime}\lbe(\le Y\!\lbe/k)\to \br_{\be{\rm t}}^{\le\prime}\lbe(\lbe X\!\times_{\be k}\be Y\!/k)
\]
is injective.
\end{remark}

We can now prove the main theorem of this paper.

\begin{theorem}\label{main} Let $S$ be a locally noetherian regular scheme and let $f\colon X\to S$ and $g\colon Y\to S$ be smooth and surjective morphisms. Assume that	
\begin{enumerate}
\item[(i)] the \'etale index $I\lbe(\le f\!\times_{\be S}\be g\le)$ is defined and is equal to $1$, 
\item[(ii)] for every point $s\in S$ of codimension $\leq 1$, the fibers $X_{\lbe s}$ and $Y_{\lbe s}$ are geometrically integral, and
\item[(iii)] for every maximal point $\eta$ of $S$,
\[
\pic\be(X_{\lbe\eta}^{\le\rm s}\be\times_{k(\eta)^{\rm s}}\be Y_{\!\eta}^{\e\rm s}\le)^{\g(\eta)}=(\e\pic X_{\lbe\eta}^{\le\rm s})^{\g(\eta)}\be\oplus\be(\e\pic Y_{\!\eta}^{\e\rm s})^{\g(\eta)}.
\]
\end{enumerate}	
Then there exists a canonical exact sequence of abelian groups
\[
\begin{array}{rcl}
0\to\brp S\overset{\!\beta^{\le 2}}{\to}\brp X\!\oplus\!\brp\, Y\overset{\!p_{\lbe X\lbe Y}^{\le 2}}{\to} \brp(X\!\be\times_{\be S}\be\lbe Y\le)&\to&\displaystyle\frac{\brxys\lbe(\lbe S\e)}{\brxs(\lbe S\e)\!\oplus\! \brys(\lbe S\e)}\\\\
&\to& \displaystyle\frac{\brtxy}{\brt\!\oplus\!\brty}\to 0,
\end{array}
\]
where $\beta^{2}$ and $p_{\be X\lbe Y}^{\le 2}$ are the maps \eqref{beta2} and \eqref{pxy2}, respectively.
\end{theorem}
\begin{proof} The theorem follows from the sequence \eqref{brob} by applying the snake lemma to the commutative diagram with exact rows 
\[
\xymatrix{0\ar[r]&\brp S\be\oplus\be\brp S\ar[rr]^(.47){(\e f^{(2)}\!,\, g^{(2)})}\ar@{->>}[d]^(.47){(\e\cdot\e)}&& \brp X\!\oplus\!\brp\, Y\ar[r]\ar[d]^(.45){p_{\be X\lbe Y}^{\le 2}}& \nbrxs\le\oplus\le\nbrys\ar@{^{(}->}[d]\ar[r]& 0\\
0\ar[r]&\brp S\ar[rr]^(.47){(\e f\lbe\times_{\be S}\le g)^{(2)}}&&\brp\lbe(X\!\be\times_{\be S}\!\lbe Y\le)\ar[r]& \nbrxys\ar[r]& 0.
}
\]
\end{proof}

\begin{remarks}\label{mcor2} \indent
\begin{itemize}
\item[(a)] The theorem yields an isomorphism $\img p_{\be X\lbe Y}^{\e 2}\simeq\cok\beta^{\le 2}$. Thus, by  Lemma \ref{b2}, $\img p_{\lbe X\lbe Y}^{\e 2}\subseteq\brp\lbe(\lbe X\!\times_{\be S}\lbe Y\le)$ is an extension of $\nbrxs\be\oplus\be\nbrys$ by $\brp S$, i.e., there exists a canonical exact sequence of abelian groups
\[
0\to \brp S\to\img p_{\lbe X\lbe Y}^{\e 2} \to\nbrxs\be\oplus\be\nbrys\to 0.
\]
\item[(b)] The proof of the theorem yields a canonical isomorphism 
\[
\cok p_{\be X\lbe Y}^{\le 2}=\displaystyle\frac{\nbrxys}{\nbrxs\be\oplus\be \nbrys}.
\]
Thus $p_{\be X\lbe Y}^{\le 2}$ is surjective if, and only if, the canonical injection $\nbrxs\be\oplus\be \nbrys\hookrightarrow\nbrxys$ is an isomorphism.
\end{itemize}
\end{remarks}

\section{Applications}

In this paper we only discuss applications of Theorem \ref{main} to smooth varieties over a field.

The following statement is a particular case of Theorem \ref{main} (recall that $\br_{\!\! X\be/k}^{\le\prime}\lbe(k\e)=(\brp\xs\le)^{\g}$ by Remark \ref{int}(c)):

\begin{proposition}\label{mcor} Let $k$ be a field and let $X$ and $\e Y$ be smooth $k$-varieties. Assume that	
\begin{enumerate}
\item[(i)] $I(X\!\times_{k}\! Y\le)=1$ and
\item[(ii)] $\pic\be(X^{\le\rm s}\be\times_{\ks}\be Y^{\e\rm s}\le)^{\g}=(\e\pic X^{\le\rm s})^{\g}\be\oplus\be(\e\pic Y^{\e\rm s})^{\g}$. 
\end{enumerate}	
Then there exists a canonical exact sequence of abelian groups
\[
\begin{array}{rcl}
0\to\br\le k\be\overset{\!\beta^{\le 2}}{\to}\be\brp X\lbe\oplus\lbe\brp\, Y\overset{\!p_{\be X\lbe Y}^{\le 2}}{\to} \brp(X\!\times_{\lbe k}\! Y\le)&\to& \displaystyle\frac{\brp(\xs\!\be\times_{\lbe \ks}\!\ys)^{\g}}{(\brp \xs)^{\g}\be\oplus\be(\brp\,\ys)^{\g}}\\\\
&\to&\displaystyle\frac{\br_{\be 2}^{\le\prime}(\be X\!\times_{k}\! Y\be/k)}{\br_{\be 2}^{\le\prime}(\be X\be/k)\!\oplus\!\br_{\be 2}^{\le\prime}(\le Y\be/k)}\to 0.
\end{array}
\]
where $\beta^{\le 2}$ is (induced by) the map \eqref{beta2} and $p_{\be X\lbe Y}^{\le 2}$ is the map \eqref{pxy2}.
\end{proposition}

\begin{remarks}\label{mcor3}\indent
\begin{enumerate}	
\item[(a)] Hypothesis (ii) of the proposition is satisfied by pairs of (possibly non-projective) smooth $k$-varieties $X, Y$ such that either $\xs$ or $\ys$ is rational. See \cite[Example 5.9(a)]{ga}.
\item[(b)] Note that the group $\brp(\be \xs\!\be\times_{\lbe \ks}\be\be\ys)^{\g}\be/\,(\brp \xs)^{\g}\be\!\oplus\! (\brp\,\ys)^{\g}$ which appears in the sequence of the proposition is canonically isomorphic to a subgroup of $\left(\brp\lbe(\be \xs\!\be\times_{\lbe \ks}\!\ys)/\brp \xs\!\oplus\be \brp\,\ys\right)^{\g}$.
\end{enumerate}
\end{remarks}

\smallskip

\begin{corollary}\label{av} Let $k$ be a field and let $X$ and $Y$ be smooth and {\rm projective} $k$-varieties. Assume that 
\begin{enumerate}
\item[(i)] $I(X\be\times_{k}\be Y\le)=1$ and
\item[(ii)] $\Hom_{\le\lle k}(A,B\le)=0$, where $A$ and $B$ are the Picard varieties of $X$ and $Y$, respectively.
\end{enumerate}
Then there exists a canonical exact sequence of abelian groups
\[
\begin{array}{rcl}
0\to\br\le k\to\be\br X\lbe\oplus\lbe\br Y\to\br\lbe(X\!\times_{\lbe k}\be Y\le)&\to& \displaystyle\frac{\br\lbe(\lbe\xs\be\times_{\lbe \ks}\! \ys)^{\g}}{(\br\lbe\xs)^{\g}\be\!\oplus\! (\br\ys)^{\g}}\\\\
&\to&\displaystyle\frac{\br_{\be 2}(\be X\!\times_{k}\! Y\be/k)}{\br_{\be 2}(\be X\be/k)\!\oplus\!\br_{\be 2}(\le Y\be/k)}\to 0.
\end{array}
\]
\end{corollary}
\begin{proof} Under the stated hypotheses, \cite[proof of Proposition 5.10]{ga} shows that 
$\pic\be(X^{\le\rm s}\be\times_{\ks}\be Y^{\e\rm s}\le)^{\g}=(\e\pic X^{\le\rm s})^{\g}\be\oplus\be(\e\pic Y^{\e\rm s})^{\g}$. Thus, in conjunction with Remark \ref{int}(d), the corollary is immediate from the proposition.
\end{proof}

\begin{remark} Hypothesis (ii) of the corollary holds in many cases of interest. See \cite[Examples 5.12]{ga}.	
\end{remark}

\begin{theorem}\label{ni} Let $k$ be a field of characteristic zero and let $X$ and $\e Y$ be smooth and projective $k$-varieties. Assume that
\begin{itemize}
\item[(i)] the geometric N\'eron-Severi groups ${\rm NS}(\xs)$ and ${\rm NS}(\e\ys)$ are torsion-free,
\item[(ii)] either $H^{1}(\xs,\s O_{\be X^{\lbe\rm s}})=0$ or $H^{1}(\e\ys,\s O_{\lle Y^{\lle\rm s}})=0$ and
\item[(iii)] $I(X\be\times_{k}\be Y\le)=1$.
\end{itemize}	
Then there exists a canonical exact sequence of abelian groups
\[
0\to\br\le k\be\overset{\!\beta^{\lle 2}}{\lra}\be\br X\be\oplus\be\br Y\overset{\! p_{\be X\lbe Y}^{\lle 2}}{\lra}\br\lbe(X\be\times_{\lbe k}\!\lbe Y\le)\to 0,
\]
where $\beta^{2}$ and $p_{\be X\lbe Y}^{\le 2}$ are (induced by) the maps \eqref{beta2} and \eqref{pxy2}, respectively.
Consequently
\begin{enumerate}
\item[(a)] every Azumaya algebra on $X\be\times_{\lbe k}\lbe Y$ is equivalent to a tensor product of pullbacks of an Azumaya algebra on $X$ and an Azumaya algebra on $Y$, and
\item[(b)] there exists a canonical exact sequence of abelian groups
\[
0\to\br\le  k\to \br\lbe\lbe(X\!\be\times_{\lbe k}\!\lbe Y\le)\to(\lle\br X\be/\e\br k)\be\oplus\be(\lle\br Y\!/\e\br k) \to 0.
\]
If $(X\times_{k} Y)(k)\neq\emptyset$, then the choice of a $k$-rational point on $X\times_{k} Y$ determines a splitting of the preceding sequence and
\[
\br\lbe\lbe(X\!\be\times_{\lbe k}\!\lbe Y\le)\simeq\br\le k\lbe\oplus\lbe (\lle\br X\!/\e\br k)\be\oplus\be(\lle\br Y\!/\e\br k).
\] 
\end{enumerate}
\end{theorem}
\begin{proof} By Remark \ref{mcor3}(b) and Corollary \ref{av}, it suffices to check that $\br\!\lbe\left(\le \xs\!\be\times_{\lbe \ks}\be\ys\e\right)=\br \xs\oplus\br\,\ys$. If $A$ and $B$ denote the Picard varieties of $X$ and $Y$, respectively, then the indicated equality follows from \cite[\S2.9, (20)]{sz} using hypothesis (i) and noting that, by (ii), either $A=0$ or $B=0$ since ${\rm dim}\, A\leq {\rm dim}_{\e k^{\lle\rm s}}\le H^{1}(\xs,\s O_{\be \xs})$ by \cite[Theorem 2.10(iii)]{fga} (and similarly for $\ys$).
\end{proof}

\begin{remark} If $k$ is a field of positive characteristic $p$ then, under hypotheses (i) and (ii) of the proposition,
the quotient $\br\!\lbe\left(\le \xs\!\be\times_{\lbe \ks}\be\ys\e\right)^{\g}\!\!/(\br \xs)^{\g}\oplus(\br\,\ys)^{\g}$ in the sequence of Corollary \ref{av} is a $p$-primary torsion abelian group. See \cite[\S2.9]{sz}.
\end{remark}

\begin{corollary}\label{rad} Let $k$ be a field of characteristic zero and let $\s C$ be the full subcategory of $({\rm Sch}\lbe/k)$ whose objects are the smooth and projective $k$-varieties $X$ such that
\begin{enumerate}
\item[(i)]  ${\rm NS}(\xs)$ is torsion-free,
\item[(ii)] $H^{1}(\xs,\s O_{\be X^{\lbe\rm s}})=0$ and
\item[(iii)] $X(k)\neq\emptyset$.
\end{enumerate}
Then $\s C$ is stable under products and the functor $\s C\to\mathbf{Ab},\e X\mapsto \br X\!/\e\br\le k$, is {\rm additive}, i.e., for every pair of objects $X,Y$ in $\s C$, the homomorphism of abelian groups induced by $p_{\be X\lbe Y}^{\lle 2}\!$ \eqref{pxy2}
\[
\br\lbe(X\!\times_{\lbe k}\! Y\le)\lbe/\e \br\le k\to(\lle\br X\be/\e\br k)\be\oplus\be(\lle\br Y\!/\e\br k)
\]
is an isomorphism. If, in addition, $\br\e k=0$, then $\br\lbe(-)$ is an additive functor on $\s C$. 
\end{corollary}
\begin{proof} Condition (iii) is clearly stable under products and (ii) is also stable by the K\"unneth formula \cite[${\rm III}_{2}$, (6.7.8.1)]{ega}. On the other hand, by (ii) and \cite[Proposition 1.7]{sz}, ${\rm NS}(\xs\!\times_{\ks}\!\ys\e)={\rm NS}(\xs)\oplus{\rm NS}(\le\ys)$ for every pair of objects $X,Y$ in $\s C$, whence (i) is stable under products as well. The second assertion of the corollary is 	immediate from the exactness of the sequence in part (b) of the theorem.
\end{proof}

\end{document}